\documentclass[letterpaper,11pt]{article}
\usepackage{amsfonts}
\usepackage{amssymb, amsmath, amsthm, graphicx, booktabs}
\usepackage[margin=1in]{geometry}
\usepackage[bottom]{footmisc}
\usepackage{natbib}
\usepackage{commath}

\DeclareMathOperator*{\argmin}{argmin}
\DeclareMathOperator*{\supp}{supp}
\newtheorem{thm}{Theorem}

\newtheorem{lemma}{Lemma}
\newtheorem{corollary}{Corollary}

\begin{document}

\title{Oracle Inequalities for Convex Loss Functions with Non-Linear Targets
}
\author{\textsc{Mehmet Caner\thanks{%
North Carolina State University, Department of Economics, 4168 Nelson Hall,
Raleigh, NC 27695. Email: \texttt{mcaner@ncsu.edu}.}} \and \textsc{Anders Bredahl Kock\thanks{%
Aarhus University and CREATES, Department of Economics, Fuglesangs Alle 4,  8210 Aarhus V, Denmark. Email: \texttt{akock@creates.au.dk}.  
}}} 
\date{\today}

\maketitle

\begin{abstract}
This paper consider penalized empirical loss minimization of convex loss functions with unknown non-linear target functions. Using the elastic net penalty we establish a finite sample oracle inequality which bounds the loss of our estimator from above with high probability. If the unknown target is linear this inequality also provides an upper bound of the estimation error of the estimated parameter vector. These are new results and they generalize the econometrics and statistics literature. Next, we use the non-asymptotic results to show that the excess loss of our estimator is asymptotically of the same order as that of the oracle. If the target is linear we give sufficient conditions for consistency of the estimated parameter vector. Next, we briefly discuss how a thresholded version of our estimator can be used to perform consistent variable selection. We give two examples of loss functions covered by our framework and show how penalized nonparametric series estimation is contained as a special case and provide a finite sample upper bound on the mean square error of the elastic net series estimator. 
\vspace{.3cm}


\end{abstract}

\section{Introduction}
Recently high-dimensional data sets have become increasingly available to researchers in many fields. In economics big data can be found in the analysis of consumer behavior based on scanner data from purchases. Furthermore, many macroeconomic variables are sampled rather infrequently leaving one with many variables compared to observations in models with many explanatory variables. Financial data is also of a high-dimensional nature with many variables and instruments being observed in small intervals due to high-frequency trading.  Alternatively, models with many variables emerge when trying to control for non-linearities in a wage regression by including basis functions of the space in which the non-linearity is supposed to be found.  Clearly, including more basis functions can result in better approximations of the non-linearity. However, this also results in a model with many variables, i.e. a high-dimensional model. For these reasons handling high-dimensional data sets has received a lot of attention in the econometrics and statistics literature in the recent years. In a seminal paper \cite{tibshirani96} introduced the Lasso estimator which carries out variable selection and parameter estimation simultaneously. The theoretical properties of this estimator have been studied extensively since then in various papers and extensions such as the adaptive Lasso by \cite{zou06}, the bridge estimator by \cite{huanghm08}, the sure independence screening by \cite{fanl08} or the square root Lasso by \cite{bellonicw11} have been proposed. For recent reviews see, e.g., \cite{fan11}, \cite{vdGB11} or \cite{bellonic11}. 

In the econometrics literature Lasso-type estimators have also proven useful. For example \cite{belloni2012sparse} have established results in the context of instrumental variable estimation without imposing the hitherto much used assumption of sub-gaussianity by means of moderate deviation theorems for self-normalized random variables. Furthermore, they allow for heteroscedastic error terms which is pathbreaking and greatly widens the scope of applicability of their results.  

Applications to panel data may be found in e.g. \cite{kock2013a}. The estimators have been studied in the context of GMM, factor models, and smooth penalties by, among others, \cite{canerz13}, \cite{canerh13}, \cite{cheng2013select} and \cite{fan2001variable}. Within linear time series models oracle inequalities have been established by \cite{kock2013oracle} and \cite{negahban12unified} have proposed a unified framework which is valid for regression as well as matrix estimation problems.

Most research has considered the linear regression model or other parametric models. In this paper we shall focus on a very general setup. In particular, we will focus on penalized empirical loss minimization of convex loss functions with potentially non-linear target functions. \cite{vdG08} studied a similar setup for the Lasso which is a special case of our results for the elastic net. Furthermore, even though our main focus is on non-asymptotic bounds, we also present asymptotic upper bounds on the excess risk and estimation error (the latter in the case where the target is linear). We also show how our results can be used to give new non-asymptotic upper bounds on penalized series estimators with many series terms. 

In particular, we
\begin{enumerate}
\item provide a finite sample oracle inequality for empirical risk minimization penalized by the elastic net penalty. This inequality is valid for convex loss functions and non-linear targets and contains an oracle inequality for the Lasso as a special case.
\item For the case where the target function is linear this oracle inequality can be used to establish finite sample upper bounds on the estimation error of the estimated parameter vector. 
\item The finite sample inequality is used to establish asymptotic results. In particular, the excess risk of our estimator is of the same order as that of an oracle which trades of the approximation and estimation errors. When the target is linear we give sufficient conditions for consistency of the estimated parameter vector.
\item In the case where the target is linear we briefly explain how a thresholded version of our estimator can unveil the correct sparsity pattern. 
\item We provide two examples of specific loss functions covered by our general framework. We verify in detail that the abstract conditions are satisfied in common settings. Then we show how nonparametric series estimation is contained as a special case of our theory and provide a finite sample upper bound on the mean square error of an elastic net series estimator. We explain why this series estimator may be more precise than classical series estimators.
\item We also note that when the loss function is quadratic the sample does not have to be identically distributed and our results are therefore also valid in the presence of heteroscedasticity.
\end{enumerate}

We stress here that our main objective is to establish upper bounds on the performance of the elastic net. It is not our intention to promote either the Lasso or the elastic net, merely to analyze the properties of the latter. However, we shall make some brief comments on merits of the two procedures when compared to each other. A clear ranking like the one in \cite{hebirivdG11} is not available at this point. However, these authors only focus on quadratic loss for which a certain data augmentation trick facilitates the analysis.

We believe that the performance guarantees on the elastic net provided by this paper are useful for the applied researcher who increasingly faces high-dimensional data sets. The usefulness is enhanced by the fact that our results are valid for a wide range of loss functions and that heteroscedasticity is allowed for when the loss function is quadratic.

The paper is organized as follows. Section \ref{setup} puts forward the setup and notation. Section \ref{secoracle} introduces the main result, the oracle inequality for empirical loss minimization of convex loss functions penalized by the elastic net. Section \ref{varsel} briefly discusses consistent variable selection by a thresholded version of the elastic net. Tuning parameter selection is handled in Section \ref{Tuning}. Section \ref{Examples} shows that the quadratic as well as the logistic loss are covered by our framework and provides an oracle inequality for penalized series estimators.

\section{Setup and notation}\label{setup}
We begin by setting the stage for general convex loss minimization. The setup is similar to the Lasso one in Section 6.3 in \cite{vdGB11}. Let $(\Omega, \mathcal{F}, P)$ be a standard probability space. Consider a sample $\cbr{Z_i}_{i=1}^n=\cbr{X_i,Y_i}_{i=1}^n$ with $X_i\in\mathcal{X}$ and $Y_i\in\mathcal{Y}\subseteq \mathbb{R}$. Here, for the sake of exposition, $\mathcal{X}$ can be thought of as a subset of $\mathbb{R}^p$ for $p\geq 1$ but as we shall see below it can be much more general. Define $\mathcal{Z}=\mathcal{X}\times\mathcal{Y}$ and let $\mathbf{F}$ be a normed real vector space with norm $\enVert[0]{\cdot}$. For each $f\in\mathbf{F}$ let $\rho_f:\mathcal{Z}\to\mathbb{R}$ be a loss function. More precisely, $f\in \mathbf{F}$ will be a function $f:\mathcal{X}\to\mathbf{R}$ and the corresponding norm will most often be the $L_2(P)$-norm $\del[0]{\int f^2(X)dP}^{1/2}$ (in fact $\enVert[0]{\cdot}$ will be the $L_2(P)$-norm in almost all our econometric examples\footnote{This choice of norm on $\mathbf{F}$ is suitable when the sample is supposed to be i.i.d.}). Furthermore, when $v$ is a vector in $\mathbb{R}^p$, $\enVert[0]{v}_1=\sum_{j=1}^p|v_j|$ denotes the $\ell_1$-norm while $\enVert[0]{v}_2=\sqrt{\sum_{j=1}^pv_j^2}$ denotes the $\ell_2$-norm. Order symbols such as $o, O, o_p$ and $O_p$ are used with their usual meanings. Also, in accordance with the usual Landau notation, $g_1(n)\in \Omega(g_2(n))$ means that there exists a constant $a>0$ such that $g_1(n)\geq ag_2(n)$ for $n$ sufficiently large. $\Theta(g_2(n))$ denotes the intersection of $O(g_2(n))$ and $\Omega(g_2(n))$ and contains all functions $g_1(n)$ that are exactly of order $g_2(n)$. Finally, for any abstract set $A$, $|A|$ denotes its cardinality. Throughout the paper we shall assume:
 
\textbf{Assumption 0:} $\cbr{Z_i}_{i=1}^n=\cbr{X_i,Y_i}_{i=1}^n$ is an independent sample and the mapping $f\mapsto \rho_f(z)$ is convex for all $z\in \mathcal{Z}$. 

The following examples provide illustrations of when the conditions in Assumption 0 are met:

\subsection*{Quadratic loss}
Let $\mathcal{X}\subset \mathbb{R}^p$ and
\begin{align*}
Y_i=f^0(X_i)+\epsilon_i,\ i=1,...,n
\end{align*}
where $\epsilon_i$ is some real error term and $f^0\in\mathbf{F}$. Then, the standard case of quadratic loss is covered by the above setting upon choosing $\rho_f(x,y)=\del[0]{y-f(x)}^2$ which is clearly convex in $f(x)$. By letting $\mathbf{F}$ only consist of linear functions $f(x)=\beta'x$ for some $\beta\in\mathbb{R}^p$ the case of linear least squares is covered. Non-linear least squares is covered by choosing $f(x)=g(\beta,x)$ for some parameter vector $\beta$. As we shall see in Section \ref{npsec} this setup can also be used to obtain some new upper bounds on nonparametric series estimation.

\subsection*{Logistic loss}
Let
\begin{align*}
Y_i^*=f^0(X_i)+\epsilon_i, i=1,...,n
\end{align*}
where $\epsilon_i$ is independent of $X_i$ and assumed to have a logistic distribution while $f^0\in\mathbf{F}$.
Assume that $Y_i=1$ if $Y_i^*> 0$ and $Y_i=0$ otherwise. Since $\epsilon_i$ has cdf $F(z)=\frac{e^z}{1+e^z}$ one gets
\begin{align*}
P(Y_i=1|X_i=x)
=
P(f^0(X_i)+\epsilon_i> 0|X_i=x)
=
\frac{e^{f^0(x)}}{1+e^{f^0(x)}}
\end{align*}
Note that for $\mathcal{X}\subseteq\mathbb{R}^p$ and $f^0(x)={\beta^0}'x$ for some parameter vector $\beta\in\mathbb{R}^p$ this is the usual expression for $P(Y_i=1|X_i=x)$ in the logit model. The above setting is more general, however, since it allows $f^0$ to be non-linear.

The log-likelihood function for a given $f\in\mathbf{F}$ is then given by (for $z=(x,y)$)
\begin{align*}
l(f|z_1,...,z_n)
=
\sum_{i=1}^n\sbr[2]{y_i\log(\frac{e^{f(x_i)}}{1+e^{f(x_i)}})+(1-y_i)(1-\frac{e^{f(x_i)}}{1+e^{f(x_i)}})}
=
\sum_{i=1}^n\sbr[2]{y_if(x_i)-\log(1+e^{f(x_i)})}
\end{align*}
Hence, a sensible loss function is the negative log-likelihood
\begin{align*}
\rho_f(x,y)=-yf(x)+\log(1+e^{f(x)})
\end{align*} 
which is convex in $f(x)$.
  
\subsection*{Negative log-likelihood}
The above two examples are both instances of the loss function being the negative of the log-likelihood. Hence, in a general setting with the negative of the log-likelihood being a convex function in $f(x)$ our results also apply. Again, a special case is $f(x)=\beta'x$.

\vspace{0.5cm}

Returning to the general setup, denote by $P_n\rho_f=\frac{1}{n}\sum_{i=1}^n\rho_f(X_i,Y_i)$ and $P\rho_f=\frac{1}{n}\sum_{i=1}^nE\rho_f(X_i,Y_i)$ the empirical and population means of the loss function for a fixed $f\in\mathbf{F}$. We shall also denote these two quantities the empirical and population risk, respectively. Note also, that in the case of identically distributed variables the population mean reduces to the plain expectation $E\rho_f(X_1,Y_1)$. We define our target as the minimizer of the theoretical risk
\begin{align*}
f^0:=\argmin_{f\in\mathbf{F}}P\rho_f
\end{align*}
where it is tacitly assumed that the minimizer exists and is unique for the $\enVert[0]{\cdot}$-norm on $\mathbf{F}$. Then, for any $f\in\mathbf{F}$, we define the excess population risk over the target as
\begin{align*}
\Xi(f):=P(\rho_f-\rho_{f^0})
\end{align*}
Note that, by construction, $\Xi(f)\geq 0$ for all $f\in\mathbf{F}$. Since the joint distribution of $(Y_i,X_i)$ is assumed to be unknown we shall consider empirical risk minimization instead of minimizing the population excess risk. Put differently, $P_n\rho_f$ is minimized. Furthermore, we will consider a linear subspace ${\cal F}_L = \{ f_{\beta}(x)=\sum_{j=1}^p\beta_j\psi_j(x),\ \beta\in \Phi\}$ of $\mathbf{F}$ where $\psi_j(x):\mathcal{X}\to\mathbb{R}$ may be thought of as basis functions of $\mathbf{F}$. Of course, in the case where $\mathcal{X}$ is a subset of $\mathbb{R}^p$, one could also think of $\psi_j(x)$ as being the $j$'th coordinate projection. This is the choice we make whenever $f^0$ is assumed to be linear. In general, $\psi(x)=(\psi_1(x),...,\psi_p(x))'$ denotes a vector of (transformed) covariates. $\Phi$ is a convex subset of $\mathbb{R}^p$ -- in many cases we can even have $\Phi=\mathbb{R}^p$. In Section \ref{Examples} we shall see an example of $\Phi=\mathbb{R}^p$ but also an example of $\Phi$ being a subset of $\mathbb{R}^p$. In case the target function $f^0$ is linear we will denote its parameter vector by $\beta^0$.

As we shall see, it is possible to prove upper bounds on the excess risk of a penalized version of the empirical risk minimizer even when $f^0$ is non-linear while we only minimize over the linear sub space $\mathcal{F}_L$
. This is non-trivial since the target belongs to a large set ($\mathbf{F}$) while we only minimize over a smaller set ($\mathcal{F}_L$). The following section gives an exact definition and discussion of our estimator.

\subsection{The elastic net}
Define 
\begin{align*}
\hat{f}
=
f_{\hat{\beta}}
=
\argmin_{f_\beta:\beta\in\Phi} \del[2]{ P_n \rho_{f_{\beta}} + \lambda_1 \| \beta \|_1 + \lambda_2 \| \beta \|_2^2 }
=
\argmin_{f\in\mathcal{F}_L} \del[2]{ P_n \rho_{f} + \lambda_1 \| \beta \|_1 + \lambda_2 \| \beta \|_2^2 }
\end{align*}
where $\lambda_1$ and $\lambda_2$ are positive constants. Hence, we are minimizing the empirical risk \textit{plus} an elastic net penalty. This form of penalty was originally introduced by \cite{zouh05} in the case of a linear regression model. The penalty is a compromise between the $\ell_1$-penalty of the plain Lasso and the squared $\ell_2$-loss in ridge regression. Ridge regression does not perform variable selection at all -- all estimated coefficients are non-zero. On the other hand, if two variables are highly correlated, the Lasso has a tendency to include only one of these. The elastic net strikes a balance between these two extremes and hence performs particularly well in the presence of highly correlated variables. This benefit has been formalized by \cite{hebirivdG11} in the case of quadratic loss. In particular they have shown that the elastic net behaves better with respect to certain restricted eigenvalue conditions than the plain Lasso.

\subsection{Assumptions}
We next turn to the assumptions needed to prove oracle inequalities for the elastic net. First, define $\mathbf{F}_{\text{local}}=\cbr[0]{f\in \mathbf{F}: \enVert[0]{f-f^0}_\infty\leq \eta}$ for some $\eta>0$ where $\enVert[0]{f}_\infty=\sup_{x\in\mathcal{X}}|f(x)|$ \footnote{Using the stronger $\enVert{\cdot}_\infty$ topology on $\mathbf{F}$ to define $\mathbf{F}_{\text{local}}$ instead of $\enVert{\cdot}$ turns out to be useful when verifying that the margin condition is satisfied with a quadratic margin in Section \ref{Examples}.}. The margin condition requires that in $\mathbf{F}_{\text{local}}\subseteq \mathbf{F}$ the excess loss $\Xi(f)$ is bounded from below by a convex function of $\enVert[0]{f-f^0}$.

\noindent{\bf Definition.}{\it We say that the {\bf margin condition} holds with strictly convex margin function $G(\cdot)$, if for all} $f \in \mathbf{F}_{\text{local}}$ {\it we have 
\[ \Xi (f) \ge G(\| f - f^0 \| ).\]}
In all examples we shall consider it can be shown that the margin condition holds for $G(u)=cu^2$ for some $c>0$ such that for all $f\in {\bf {F}}_{\text{local}}$, $\Xi(f)\geq c\| f - f^0 \|^2 $. More generally, we present a sufficient condition for $G$ to be quadratic in Section \ref{Examples}. The convex conjugate of $G(\cdot)$ will also play a role in the development of the oracle inequalities below. In particular, the following definition is taken from page 121 in \cite{vdGB11} and many more properties of convex conjugates can be found in \cite{rockafellar97}.

\noindent{\bf Definition}. {\it Let $G$ be a strictly convex function on $[0, \infty)$ with $G(0)=0$. The {\bf convex conjugate $H$} of $G$ is defined as
\[ H(v)= \sup_{u\geq 0} \{ u v - G(u) \}, \quad v \ge 0,\]}
Lemma \ref{Hprop} in the appendix establishes some properties of $H(v)$. Note also that if $G(u)=cu^2$, then $H(v)=v^2/(4c)$. Furthermore, from the definition of the convex conjugate
\begin{equation}
 u v \le G(u) + H(v).\label{fi}
\end{equation}
which is also known as Fenchel's inequality. Next, for any subset $S$ of $\cbr[0]{1,...,p}$ and $\beta\in \mathbb{R}^p$ we define $\beta_S$ such that $\beta_{S,j}=\beta_j1_{\cbr[0]{j\in S}}$ for $j=1,...,p$. Letting $|S|=s$ denote the cardinality of $S$ we may define

\noindent{\bf Definition}. {\it The {\bf adaptive restricted eigenvalue condition} is satisfied with if 
\begin{align*}
\phi^2(S)=\min\cbr[3]{\frac{\enVert[0]{f_\beta}^2}{\enVert[0]{\beta_S}_2^2} : \beta\in\mathbb{R}^p\setminus \cbr[0]{0}, \|\beta_{S^c} \|_1 \le L_n  \| \beta_S \|_2 }>0
\end{align*}
where $L_n= 3(\sqrt{|S|}+ \frac{2 \lambda_2 \|\beta_S\|_2}{\lambda_1})$}. As mentioned already, in many econometric examples one may choose $\enVert[0]{\cdot}$ to be the $L_2(P)$-norm. In this case, and if the covariates are also identically distributed (as will be assumed in our concrete examples in Section \ref{Examples}), $\enVert[0]{f_\beta}^2=\beta'\Sigma\beta$ where $\Sigma=E(\psi(X_1)\psi'(X_1))$. Hence, 
\begin{align}
\phi^2(S)=\min\cbr[3]{\frac{\beta'\Sigma\beta}{\enVert[0]{\beta_S}_2^2} : \beta\in\mathbb{R}^p\setminus \cbr[0]{0}, \|\beta_{S^c} \|_1 \le L_n  \| \beta_S \|_2 }\label{reseig}
\end{align}
Since
\begin{align*}
\phi^2(S)
\geq
\min_{\beta\in\mathbb{R}^p}\frac{\beta'\Sigma\beta}{\enVert[0]{\beta_S}_2^2}
\geq
\min_{\beta\in\mathbb{R}^p}\frac{\beta'\Sigma\beta}{\enVert[0]{\beta}_2^2}
\end{align*}
for all $S\in\cbr[0]{1,...,p}$ the adaptive restricted eigenvalue condition is satisfied in particular when the smallest eigenvalue of the population covariance matrix is positive. However, since the minimum in (\ref{reseig}) is taken over a subset of $\mathbb{R}^p$ only, we may have $\phi(S)>0$ even when $\Sigma$ is singular. Note also that the adaptive restricted eigenvalue condition is used various guises in the literature and is similar to the eigenvalue conditions of \cite{bickelrt09} and \cite{hebirivdG11}.

Before defining what we understand by the oracle estimator, define $S_\beta=\cbr[0]{j:\beta_j\neq 0}$ as the subset of $\cbr[0]{1,...,p}$ containing the indices of the non-zero coefficients. Let $s_\beta=|S_\beta|$ denote the cardinality of this set. Then, letting $\Gamma$ denote a collection of subsets of $\cbr[0]{1,...,p}$, we define

\noindent {\bf Definition}  {\it The oracle estimator is defined as}
\begin{align}
\beta^* = \argmin_{\beta:S_{\beta} \in \Gamma}  \cbr[4]{3 \Xi (f_{\beta}) + 2 H \left(\frac{4 \lambda_1 \sqrt{s_{\beta}} +4 \lambda_2 \|\beta\|_2}{\phi (S_{\beta})}\right)}\label{DefOracle}.
\end{align}
Note that the definition of the oracle still leaves considerable freedom since $\Gamma$ is defined by the user -- a property which we shall utilize later when considering linear targets (see e.g. remark 2 after Theorems \ref{thm1} below). In the case where $\Gamma$ equals the power set of $\cbr[0]{1,...,p}$ the oracle estimator may equivalently be written as
\begin{align*}
\beta^* = \argmin_{\beta\in \mathbb{R}^p}  \cbr[4]{3 \Xi (f_{\beta}) + 2 H \left(\frac{4 \lambda_1 \sqrt{s_{\beta}} +4 \lambda_2 \|\beta\|_2}{\phi (S_{\beta})}\right)}.
\end{align*}  
The definition of the oracle in (\ref{DefOracle}) turns out to be convenient for technical reasons but it also has a useful interpretation as a tradeoff between approximation and estimation error: In the standard setting of a quadratic loss function with a linear target, i.e. $f^0(x)=x'\beta^0$, it is known that the squared $\ell_2$-estimation error of Lasso type estimators when estimating $p$ parameters, of which $s$ are non-zero, is of the order $\frac{s\log(p)}{n}$. In the case of quadratic loss in the beginning of this section one has $\Xi(f)=E(f(X_1)-f^0(X_1))^2$ if the sample is identically distributed and $X_1$ and $\epsilon_1$ are independent. So $G$ and hence $H$ are quadratic in the definition of the margin condition. Choosing $\lambda_2=\frac{\lambda_1\sqrt{s}}{2\enVert[0]{\beta}}$ and $\lambda_1$ of the order $\sqrt{\log(p)/n}$, which are both choices we shall adhere to in the sequel, one finds that $H(\cdot)$ is of the order $\frac{s\log(p)}{n}$. This is exactly the estimation error under quadratic loss and motivates coining $H(\cdot)$ the estimation error term. The term $\Xi(f_\beta)$ is referred to as the approximation error and (\ref{DefOracle}) shows that the oracle trades of these two terms: a lower approximation error can be obtained by increasing $s_\beta$ while this also implies estimating more parameters resulting in a higher estimation error.    

Finally, letting $S^*=S_{\beta^*},\ \phi_*=\phi(S^*),\ s^*=|S^*|$ and $f^*=f_{\beta^*}$ we denote the oracle bound (value of the objective function minimized by the oracle) by
\begin{align*}
2 \Delta^* := 3 \Xi (f^*) + 2 H \left(\frac{4 \lambda_1 \sqrt{s^*} + 4 \lambda_2 \|\beta^*\|_2}{\phi_*}\right). 
\end{align*}

The inequality in Theorem \ref{thm1} below will be valid on a random set which we introduce next. In Theorem \ref{pbound} we shall show that this set actually has a high probability by means of a suitable concentration inequality for suprema of empirical processes. Define the empirical process
\begin{align*}
\cbr[2]{V_n (\beta) = \frac{1}{n} \sum_{i=1}^n [\rho_{f_{\beta}} (Z_i) - E \rho_{f_{\beta}} (Z_i)],\ \beta\in\mathbb{R}^p}.
\end{align*}
Next, we introduce a local supremum of the empirical process in incremental form
\begin{equation}
 Z_{M} = \sup_{\| \beta - \beta^* \|_1 \le M} | V_n (\beta ) - V_n (\beta^*) |.\label{2}
\end{equation}
Then we define 
\[ M^* = \Delta^*/\lambda_0,\]
where $\lambda_0$ is a positive sequence and set
\begin{align*}
\tau = \{ Z_{M^*} \le \lambda_0 M^* \} = \{ Z_{M^*} \le \Delta^* \}.
\end{align*}
The set $\tau$ is the one we shall work on in Theorem \ref{thm1} below. Note in particular, that on $\tau$, $Z_M$ can not be larger than $\Delta^*$ which is the minimal value of the loss function  of the oracle.   

We are now ready to state our assumptions:

{\bf Assumption 1}.{\it Assume the margin condition  with strictly convex function $G(\cdot)$}.

{\bf Assumption 2}.{\it Assume that} $f^* \in \mathbf{F}_{\text{local}}$ \textit{and} $f_{\beta} \in  \mathbf{F}_{\text{local}}$ {\it for all $\|\beta - \beta^*\|_1 \le M^*$.} 

{\bf Assumption 3}.{\it Assume that adaptive restricted eigenvalue condition holds for $S^*$, i.e. $\phi(S^*)>0$.}

As discussed above, the margin condition, Assumption 1, regulates the behavior of the excess risk function. When $\mathbf{F}$ is equipped with the $L_2(P)$  we will see in Section \ref{Examples} that the margin condition is actually often satisfied with $G(\cdot)$ being quadratic. Put differently, the margin condition is satisfied in many examples with a quadratic margin.

Assumption 2 is a technical condition which enables us to use the margin condition for $f^*$. The first part requires that the oracle is a good approximation to $f^0$ in the sup-norm. Of course the validity of this statement depends on how well linear combinations of elements in $\cbr[0]{\psi_j}_{j=1}^p$ can approximate $f^0$. The validity also depends on the choice of $\Gamma$ in the definition of the oracle since the precise form of $f^*$ depends on this. For concrete choices of $\mathbf{F}$ one can make proper choices of bases that guarantee the desired degree of approximation. Note in particular, that it follows from remark 2 in Section \ref{secoracle}, that when $f^0$ is linear one can choose $\Gamma$ such that $f^*=f^0$ and so, a fortiori, $f^*\in\mathbf{F}_{\text{local}}$. The second part of Assumption 2 states that $f_\beta\in\mathbf{F}$ if $\beta$ is close to $\beta^*$. This is rather innocent by the triangle inequality. We will give more detailed sufficient conditions for Assumption 2 in Section \ref{Examples} for concrete econometric examples. In particular, if $\mathbf{F}$ consists of sufficiently smooth functions\footnote{To be concrete, we shall be considering a H\"{o}lder class of function to be defined precisely in Section \ref{Examples}.}, we shall exhibit concrete choices of bases $\cbr[0]{\psi_j}_{j=1}^\infty$ and collections of sets $\Gamma$ such that $f^*$ approximates $f^0$ to the desired degree. Assumption 3 has been discussed above and is valid when $\enVert{\cdot}$ is the $L^2(P)$-norm and $\Sigma=E(\psi(X_1)\psi'(X_1))$ has full rank.

\section{An Oracle Inequality}\label{secoracle}
In this section we extend Theorem 6.4 of \cite{vdGB11} from $\ell_1$ penalty (Lasso) to $\ell_1 + \ell_2^2$ penalty (Elastic Net). This is not a trivial extension since the basic inequality used to establish the result has to be altered considerably. More precisely, the inequality that ties the estimator to the oracle has to be modified. The second difference is that we need to use the adaptive restricted eigenvalue condition which is different from the compatibility condition used in the $\ell_1$-case.  Compared to the linear target with quadratic loss in \cite{hebirivdG11} estimated by the elastic net our proof cannot benefit from the augmented regressors idea since this idea relies crucially on the loss function being quadratic. In the case of general convex loss function the proof technique is entirely different and we use the margin condition, Fenchel's inequality, and a careful definition of the oracle instead. We would like to stress that Theorem \ref{thm1} below is purely deterministic in the sense that there are no probabilities attached to it. It is valid on the set $\tau$ to which we shall later attach a lower bound on its probability. It also
provides a finite sample result -- i.e. the result is valid for any sample size and not just asymptotically. 
\begin{thm}\label{thm1}
Suppose 
$\lambda_1$ satisfies $\lambda_1 \ge 8 \lambda_0$. Then on the set $\tau$, under Assumptions 1-3, we have 
\[ \Xi (\hat{f}) + \lambda_1 \| \hat{\beta} - \beta^* \|_1 + \lambda_2 \|\hat{\beta}_{S^*} - \beta_{S^*}^*\|_2^2\le 
4\Delta^*
=
6 \Xi (f^*) + 4  H \left( \frac{4 \lambda_1 \sqrt{s^*} +4 \lambda_2 \|\beta^*\|_2}{\phi_*} \right),\]
where $H(\cdot)$ is the convex conjugate of the function $G(\cdot)$ in the margin condition.
\end{thm}

Note that Theorem \ref{thm1} provides an upper bound, $4\Delta^*$, on the excess loss of $\hat{f}$ in terms of the excess loss of the oracle $f^*$ as well as an extra term $H(\cdot)$, the estimation error, which is hopefully not too big. We shall comment much more on this extra term in the sequel. Theorem \ref{thm1} can also be used to give an upper bound on the $\ell_1$-estimation error. Due to its importance the theorem warrants some detailed remarks.
  
1. The result of Theorem \ref{thm1} reduces to the result for the Lasso in Theorem 6.4 of  \cite{vdGB11} when we set $\lambda_2=0$ except for the fact that our adaptive restricted eigenvalue condition is slightly stronger than their compatibility constraint. In that sense we generalize the oracle inequality of \cite{vdGB11}. Their oracle inequality is, with $\phi_{**}$ being a compatibility constant (see p.157, \cite{vdGB11}) 
\[\Xi (\hat{f}) + \lambda_1 \| \hat{\beta} - \beta^* \|_1 + \le 6 \Xi (f^*) + 4  H \left( \frac{4 \lambda_1 \sqrt{s^*} }{\phi_{**}} \right)\label{vdGB6.4}.\]
As mentioned, the only difference between their Theorem 6.4 and the result that can be deduced from our Theorem \ref{thm1} is that $\phi_{**}\geq\phi_{*}$. However, we also carried out the proofs of Theorem \ref{thm1} imposing the compatibility constraint instead of the adaptive restricted eigenvalue condition and Theorem \ref{thm1} then reduced to Theorem 6.4 in (\cite{vdGB11}) upon setting $\lambda_2=0$. This result can be obtained from the authors on demand. The reason that the adaptive restricted eigenvalue condition is used in the general elastic net is that it gives sharper bounds than the compatibility condition in the general elastic net case. More precisely, if the compatibility condition is used in our case we get an extra $\sqrt{s^*}$ term in front of $4 \lambda_2 \|\beta^*\|_2$ in the function $H(\cdot)$. 

2. Letting $\beta^{BL}$ denote $\argmin_{\beta}\rho_{f_\beta}$, i.e. the best linear approximation, and setting $\Gamma=S_{BL}=\cbr[0]{j:\beta^{BL}_j\neq 0}$ and choosing $\lambda_2=\frac{\lambda_1\sqrt{|\beta_{BL}|}}{2\enVert{\beta^{BL}}_2}$ it follows that 
\begin{align*}
\beta^* = \argmin_{\beta\in S_{BL}} \cbr[4]{ 3 \Xi (f_{\beta}) + 2 H \left(\frac{6 \lambda_1 \sqrt{|S_{BL}|}}{\phi (S_{BL})}\right)}
\end{align*}
Note how we have used our discretion in making a choice of $\Gamma$ which will turn out to be useful below. Since the second term in the definition of $\beta^*$ does not depend on $\beta$ in this case it follows that $\beta^*$ is the minimizer of $\Xi(f_\beta)$ which itself also is the minimizer of $\rho_{f_\beta}$. Hence, $\beta^*=\beta^{BL}$ in this case. It follows that under the conditions of Theorem \ref{thm1}
\begin{align}
\Xi (\hat{f}) + \lambda_1 \| \hat{\beta} - \beta^{BL} \|_1 + \lambda_2 \|\hat{\beta}_{S_{BL}} - \beta^{BL}_{S_{BL}}\|_2^2
\leq  
6 \Xi (f_{\beta^{BL}}) + 4  H \left( \frac{6\lambda_1 \sqrt{|S_{BL}|}}{\phi(S_{BL})} \right)\label{linapp}
\end{align}
So, in particular, Theorem \ref{thm1} can be used to provide upper bounds on the $\ell_1$-distance of $\hat{\beta}$ to $\beta^{BL}$ due to our freedom in defining the oracle $\beta^*$. If the target $f^0$ is also linear then clearly the best linear approximation equals the target implying that $\beta^{BL}=\beta^0$ and hence $f_{\beta^{BL}}=f^0$ and $\Xi(f_{\beta^{BL}})=\Xi(f^0)=0$. Using $S_{BL}=S^0=\cbr[0]{j:\beta^{0}_j\neq 0}$, inequality (\ref{linapp}) yields
\begin{align}
\Xi (\hat{f}) + \lambda_1 \| \hat{\beta} - \beta^0 \|_1 + \lambda_2 \|\hat{\beta}_{S^{0}} - \beta^0_{S^{0}}\|_2^2
\leq  
4  H \left( \frac{6\lambda_1 \sqrt{|S^{0}|}}{\phi(S^0)} \right).
\label{linapp2}
\end{align}
Hence, in case the target is linear, (\ref{linapp2}) in particular yields an upper bound on the $\ell_1$-estimation error which does not depend on the excess loss of the oracle. We shall make use of this fact in Section \ref{varsel} on variable selection. It is also worth pointing out that in practice one does not know $S^0$ and hence can't choose $\Gamma=S_{BL}=S^0$ in the development of (\ref{linapp2}). However, (\ref{linapp2}) is valid even without this knowledge since an even sharper upper bound follows from Theorem \ref{thm1} by choosing $\Gamma=\cbr[1]{A\subset\cbr[0]{1,...,p}: |A|\leq |S^0|}$, i.e subsets of $\cbr[0]{1,...,p}$ of cardinality at most $|S^0|$. This bound only relies on sparseness of the target and since $S^0$ is a member of $\Gamma$ (\ref{linapp2}) follows a fortiori.  
 
3. A key issue is to understand the effect of $\lambda_2$ on the excess risk. Clearly, the right hand side of Theorem \ref{thm1} is increasing in $\lambda_2$ through $H(\cdot)$ (by Lemma \ref{Hprop} in the appendix $H(\cdot)$ is non-decreasing). However, the same is the case for the left hand side through its multiplication onto the squared $\ell_2$-error. This illustrates a tradeoff in the size of $\lambda_2$.

4. Note also, that the very definition of the restricted set in the definition of the adaptive restricted eigenvalue condition also depends on $\lambda_2$ through $L_n$. In particular, increasing $\lambda_2$ increases the size of the set we are minimizing over in the definition of $\phi(S)$. This implies that choosing $\lambda_2$ too large may lead to $\phi(S)=0$, or at least undesirably small values of $\phi(S)$. Note that choosing $\lambda_2 = \frac{\lambda_1 \sqrt{s^*}}{2\| \beta^*\|_2}$ results in $L_n = 6 \sqrt{s^*}$.
So in this case the size of the restricted set only depends on the cardinality of the oracle. Here it is worth noticing that the sparser the oracle ($s^*$ small) the smaller will the restricted set be and the larger will $\phi_*$ be. 

5. In many econometric examples the margin condition (Assumption 1) is satisfied with a quadratic margin resulting in $H(v)=v^2/4c$ for a positive constant $c$ (as argued just after the definition of the convex conjugate). Setting $\lambda_2 = \frac{\lambda_1 \sqrt{s^*}}{2\| \beta^*\|_2}$ in Theorem \ref{thm1} results in
\[ \Xi (\hat{f}) + \lambda_1 \| \hat{\beta} - \beta^* \|_1 +\lambda_1 \frac{\sqrt{s^*}}{2\|\beta^*\|_2} \|\hat{\beta}_{S^*} - \beta_{s^*}\|_2^2\le 6 \Xi (f^*) 
+  \frac{36 \lambda_1^2 s^*}{c \phi_*^2}.\]
Note that for $\lambda_2=0$, corresponding to a pure $\ell_1$-loss, Theorem \ref{thm1} reduces to 
\[\Xi (\hat{f}) + \lambda_1 \| \hat{\beta} - \beta^* \|_1 + \leq 6 \Xi (f^*) +  \frac{16 \lambda_1^2 s^*}{c \phi_*^2}.\]

 
Recall that Theorem \ref{thm1} and the remarks following it are valid on the set $\tau = \cbr[0]{ Z_{M^*} \le \lambda_0 M^*}= \cbr[0]{ Z_{M^*} \le \Delta^*}$. As a consequence, we would like $\tau$ to have a large probability. This can be achieved by choosing $\lambda_0$ large. However, note that Theorem \ref{thm1} supposes $\lambda_1\geq 8\lambda_0$ such that the right hand side of Theorem \ref{thm1} is also increasing in $\lambda_0$. Put differently, there is a tradeoff between the tightness of the bound in Theorem \ref{thm1} and the probability with which the bound holds. In the following we shall give a lower bound on the probability of $\tau$ which trades off these two effects.

Assume that $\rho_{f}(x,y) = \gamma (y, f (x)) +c (f)$ where $c(f)$ is a constant possibly depending on $f$. It will always be zero in our examples. We further assume that there exists a $D>0$ such that
\begin{align}
|\gamma (y, f_\beta(x))-\gamma (y, f_{\tilde{\beta}} (x))| \leq D|f_\beta(x)-f_{\tilde{\beta}}(x)|\label{lipschitz}
\end{align} 
for all $(x,y)\in \mathcal{X}\times \mathcal{Y}$ and $\beta,\tilde{\beta}\in \Phi$. In other words $\gamma(\cdot,\cdot)$ is assumed to be Lipschitz continuous in its second argument over $\mathcal{F}_L$ with Lipschitz constant $D$. The reason we only need Lipschitz continuity over $\mathcal{F}_L$ is that it is used for a contraction inequality in connection with bounding the local supremum of the empirical process $Z_M$ in (\ref{2}). Assume furthermore that
\begin{equation}
 \frac{1}{n} \sum_{i=1}^n \max_{ 1 \le j \le p} E \psi_j^2 (X_i) \le 1,\label{4.2}
\end{equation}
and for a positive constant $K$
\begin{equation}
 \max_{1 \le j \le p} \| \psi_j  \|_{\infty} \le K.\label{4.3}
\end{equation}
Note that (\ref{4.3}) implies (\ref{4.2}) when $K=1$. Assuming $ \max_{1 \le j \le p} \| \psi_j  \|_{\infty} \le K$ is rather innocent since many commonly used basis functions are bounded. As we shall see in Section \ref{Examples} the most critical assumption in concrete examples is the Lipschitz continuity of the loss function. With this notation in place we state the following result which builds on Theorem 14.5 \cite{vdGB11} (see also Corollary A.1 in \cite{vdG08}).

\begin{thm}\label{pbound}
Assume that (\ref{lipschitz})-(\ref{4.3}) are valid and that $p\geq 2$. Assume furthermore, that $\log(p)\leq n$. Then, there exists a constant $d>0$ such that choosing $\lambda_0=dD\sqrt{\frac{\log(p)}{n}}$ yields
\begin{align*}
P(\tau)\geq 1-\del[2]{\frac{1}{p}}
\end{align*} 
\end{thm}

The assumption $p\geq 2$ is made for purely technical reasons and does not exclude any interesting problems. Similarly, $\log(p)\leq n$ still allows $p$ to increase at an exponential rate in the sample size.\footnote{This is from an asymptotic point of view, though we wish to emphasize that the inequality in Theorem \ref{pbound} holds for any given sample size (satisfying the conditions of the theorem).} From an asymptotic point of view Theorem \ref{pbound} reveals that the measure of the set $\tau$, on which the inequality in Theorem \ref{thm1} is valid, tends to 1 as $p\to\infty$. In order to also cover the case of fixed $p$ one can choose  $\lambda_0=dD\sqrt{\frac{\log(p)\log(n)}{n}}$ in Theorem \ref{pbound} to obtain $P(\tau)\geq 1-\del[2]{\frac{1}{p}}^{\log(n)}$ by a slight modification of the proof of Theorem \ref{pbound} which tends to one as $n\to\infty$ even for $p$ fixed. But since Theorem \ref{thm1} requires $\lambda_1\geq 8\lambda_0$ the this will yield a bigger upper bound in that Theorem since by Lemma \ref{Hprop} in the Appendix $H(\cdot)$ is a non-decreasing function. Combining Theorems \ref{thm1} and \ref{pbound} yields the following result.
\begin{thm}\label{thm3}
Under the assumptions of Theorems \ref{thm1} and \ref{pbound} with the choice of $\lambda_0=dD\sqrt{\frac{\log(p)}{n}}$ for some positive constant $d$ it holds with probability at least $1-\del{\frac{1}{p}}$
\begin{align*}
\Xi (\hat{f}) + \lambda_1 \| \hat{\beta} - \beta^* \|_1 + \lambda_2 \|\hat{\beta}_{S^*} - \beta_{S^*}^*\|_2^2\le 
4\Delta^*
=
6 \Xi (f^*) + 4  H \left( \frac{4 \lambda_1 \sqrt{s^*} +4 \lambda_2 \|\beta^*\|_2}{\phi_*} \right)
\end{align*}
where $H(\cdot)$ is the convex conjugate of $G(\cdot)$.
\end{thm}
Theorem \ref{thm3} basically consists of the inequality in Theorem \ref{thm1} with a lower bound attached to the measure of the set $\tau$ on which Theorem \ref{thm1} is valid. In particular, the theorem reveals that the excess loss of, $\Xi(\hat{f})$, will not be much larger than the one of the oracle. The second term on the right hand side reflects the estimation error. Put differently, the excess loss of our estimator depends on the excess loss of the oracle as well as the distance to the oracle. From an asymptotic point of view Theorem \ref{pbound} reveals that the measure of the set $\tau$ on which the inequality in Theorem \ref{thm1} is valid tends to 1 as $p\to\infty$. In particular, we have the following result for the asymptotic excess loss of $\hat{f}$. To this end assume that $p\in O(\exp(n^a))$ and $|S^*|\in O(n^b)$ for some $a>0$ and $b\geq 0$.

\begin{corollary}\label{corloss}
Assume that $\lambda_1\leq L\lambda_0$ for some $L\geq 8$. Then, under the assumptions of Theorems \ref{thm1} and \ref{pbound}, with $\lambda_2 = \lambda_1 \sqrt{s^*}/2 \|\beta^*\|_2$  and  the choice of $\lambda_0=dD\sqrt{\frac{\log(p)}{n}}$ for some positive constant $d$ one has
\begin{align*}
\limsup_{n\to\infty} \Xi (\hat{f})\leq 6 \limsup_{n\to\infty}\Xi (f^*)
\end{align*}
with probability approaching one if $a+b<1$ and $\phi_*$ is bounded away from zero. 
\end{corollary}
Corollary \ref{corloss} shows that asymptotically the excess loss of $\hat{f}$ will be of the same order as that of the oracle. This is useful since we saw in remark 2 above that we have considerable discretion in choosing $f^*$ and hence in what we bound $\Xi(\hat{f})$ by from above in Corollary \ref{corloss}. In the case where $f^0$ is linear we know from Remark 2 above that we can choose $\Gamma$ such that $f^*=f^0$ and hence $\Xi(f^*)=0$. In this case Corollary \ref{corloss} actually reveals that the excess loss of $\hat{f}$ tends to zero. We next investigate the case of linear $f^0$ in more detail.

\subsection{Linear target}\label{lintar}
In the case where the target function $f^0$ is linear Theorem \ref{thm3} can be used to deduce the following result.

\begin{corollary}\label{corlin}
Assume that $f^0$ is linear. 

a) Then, under the assumptions of Theorems \ref{thm1} and \ref{pbound}
\begin{align}
\Xi (\hat{f})
&\leq 
4H \left( \frac{4 \lambda_1 \sqrt{|S^0|} +4 \lambda_2 \|\beta^0\|_2}{\phi (S^0)} \right)\label{loss1}\\
 \| \hat{\beta} - \beta^0\|_1
&\leq 
\frac{4}{\lambda_1}H \left( \frac{4 \lambda_1 \sqrt{|S^0|} +4 \lambda_2 \|\beta^0\|_2}{\phi (S^0)} \right)\label{lin1}
\end{align}
with probability at least  $1-\del{\frac{1}{p}}$.

b) If, furthermore, the margin is quadratic such that $H(v)=v^2/(4c)$ for some $c>0$, and we choose $\lambda_2 = \frac{\lambda_1 \sqrt{|S^0|}}{2\| \beta^0\|_2}$ as well as $\lambda_1\leq L\lambda_0$ for some $L\geq 8$, then, by (\ref{loss1}) and (\ref{lin1}) 
\begin{align}
\Xi (\hat{f})
\leq 
\frac{36}{c}\frac{\lambda_1^2|S^0|}{\phi^2(S^0)}
\leq 
\frac{36}{c}\frac{(LdD)^2\log(p)|S^0|}{n\phi^2(S^0)}\label{loss2}\\
\| \hat{\beta} - \beta^0 \|_1
\leq 
\frac{36L}{c}\frac{\lambda_1|S^0|}{\phi^2(S^0)}
\leq 
\frac{36LdD}{c} \sqrt{\frac{\log(p)}{n}}\frac{|S^0|}{\phi^2(S^0)}\label{lin2}
\end{align}
with probability at least  $1-\del{\frac{1}{p}}$.
\end{corollary}
The bounds (\ref{loss1}) and (\ref{lin1}) bound the $\ell_1$-estimation error of the elastic net estimator for any type of loss function satisfying the conditions of Theorem \ref{thm1}. Note that there is no excess loss from the oracle entering in the upper bound. This is due to the fact that this is zero when the target is linear. The last two bounds bounds in Corollary \ref{corlin} specialize to the case where the quadratic margin condition is satisfied. We stress again, as we shall see later (see Section \ref{Examples}), that the margin condition is indeed quadratic in many econometric examples. 

Furthermore, one sees from the above Corollary that the rate of convergence of the elastic net estimator in the $\ell_1$-norm is $\frac{\sqrt{n}}{\sqrt{\log(p)}}\frac{1}{|S^0|}$provided that the adaptive restricted eigenvalue is bounded away from zero. 

Furthermore, (\ref{lin2}) can be used to deduce consistency of $\hat{\beta}$ for $\beta^0$. As in Corollary \ref{corloss} assume that $p =O(\exp(n^a))$ and $|S^0| = O(n^b)$ for some $a>0$ and $b\geq 0$.

\begin{corollary}\label{corlinasym}
Assume that $f^0$ is linear and set $\lambda_2 = \frac{\lambda_1 \sqrt{|S^0|}}{2\| \beta^*\|_2}$ as well as $\lambda_1\leq L\lambda_0$ for some $L\geq 8$. Let the quadratic margin condition be satisfied and assume furthermore that $\phi^2(S^0)$ is bounded away from 0. Then, under the assumptions of Theorems \ref{thm1} and \ref{pbound},
\begin{align*}
 \| \hat{\beta} - \beta^0 \|_1 \stackrel{p}{\to} 0 
\end{align*}
if $a+2b<1$.
\end{corollary}  
Corollary \ref{corlinasym} shows that the elastic net can be consistent even when the dimension $p$ increases at a subexponential rate in the sample size. Note, however, that the number of relevant variables, $|S^0|$ can not increase faster than the square root of the sample size ($a$ can be put arbitrarily close to 0 to see this). Hence, even though the total number of variables can be very large, the number of relevant variables must still be quite low. This is in line with previous findings for the linear model in the literature. We also remark that this requirement is slightly stricter than the one needed when considering the excess loss in Corollary \ref{corloss} (in that corollary we only needed $a+b<1$). Also note that the conditions in Corollary \ref{corlinasym} are merely sufficient. For example one can let $\phi^2(S^0)$ tend to zero at the price of reducing the growth rate of $p$ and $|S^0|$.

\section{Variable Selection}\label{varsel}
In this section we briefly comment on how the results in Section \ref{secoracle} can be used to perform consistent variable selection in the case where $f^0$ is a linear function\footnote{If the target function is not linear we do not find it sensible to talk about consistent variable selection in a linear approximation of the target. Hence, this section restricts attention to the case where the target is linear.}. First note, that the results in Corollaries \ref{corlin} and \ref{corlinasym} can be used to provide rates of convergence of $\hat{\beta}$ for $\beta^0$ in in the $\ell_1$-norm in the case of a linear target. If one furthermore assumes that $\min\cbr[0]{|\beta_j^0|:\beta_j^0\neq 0}$ is bounded away from zero by at least the rate of convergence of $\hat{\beta}$ it follows by standard arguments, see \cite{lounici08} or \cite{kock2013oracle}, that no non-zero $\beta^0$ will be classified as such. Put differently, the elastic net possesses the \textit{screening property}.

In order to remove all non-zero variables one may furthermore threshold the elastic net estimator by removing all variables with parameters below a certain threshold. Again standard arguments show that choosing the threshold of the order of the rate of convergence (details omitted) can yield consistent model selection asymptotically. Since thresholding is a generic technique which is not specific to our setup we shall not elaborate further on this at this stage.    

One technical remark is in its place at this point. Since thresholding is done at the level of the individual parameter, what one really needs is an upper bound on the estimation error for each individual parameter. In other words, an upper bound on the sup-norm, $\max_{1\leq j\leq p}|\hat{\beta}_j-\beta_j^0|$, is sought. However, our results in the previous section provide rates of convergence in the much stronger $\ell_1$-norm\footnote{On finite-dimensional vector spaces these two norms are equivalent but here we are working in a setting where the dimension, $p$, tends to infinity.}. Of course one may simply use the $\ell_1$-rates of convergence to upper bound the sup-norm rates of convergence. But this is suboptimal. Alternatively, a strengthening of the adaptive restricted eigenvalue condition can yield rates of convergence in the $\ell_2$- or the sup-norm which are of a lower order of magnitude than the corresponding $\ell_1$-results. Upon request we can make results for the thresholded elastic net based on upper bounds on the $\ell_2$-norm rate of convergence available. We have omitted the results here since thresholding is a rather standard technique.

\section{Tuning Parameter Selection}\label{Tuning}
Recently, \cite{fant13} developed a method to select tuning parameters in high dimensional generalized linear models with more parameters than observations. Here we briefly describe their method with the terminology translated into our setting\footnote{\cite{fant13} consider log-likelihood functions of generalized linear models. However, as our loss functions can often be written as the negative of the log-likelihood (as seen in Section \ref{setup}), their setup applies to many of our examples.}. For a loss function of the form 
\begin{align}
\rho_\beta(Z_i) = -Y_i\beta'X_i+b (\beta'X_i)-c (Y_i, \xi)\label{GIC},
\end{align}
where $b(\cdot)$ and $c(\cdot, \cdot)$ are known functions and $\xi$ is a known scale parameter. The tuning parameter, $\lambda$, (as a generic tuning parameter) is chosen to minimize the Generalized Information Criterion (GIC) 
\begin{align*}
GIC (\lambda) = 2 \sbr[3]{\frac{1}{n} \sum_{i=1}^n \rho_{\hat{\beta}_\lambda} (Z_i)-\frac{1}{n} \sum_{i=1}^n \rho_{\text{saturated}} (Z_i)} + \frac{\log (\log (n))\log(p)}{n}s_{\lambda}. 
\end{align*}
where $\rho_{\text{saturated}}$ corresponds to the saturated model as defined in \cite{fant13}  and $\hat{\beta}_\lambda$ is the elastic net minimizer corresponding to the penalty parameter $\lambda$. $s_{\lambda}$ is the number of non-zero coefficients for a given value of $\lambda$.

Theorems 1-2 and Corollary 1 in \cite{fant13} establish the consistency of this approach, i.e. $P\del[1]{\supp(\hat{\beta}_\lambda)=\supp(\beta^0)}\to 1$ where $\supp(v)$ indicates the support of the vector $v$, i.e. the location of its non-zero entries. Hence, GIC will, asymptotically, select the correct model (provided there exists a $\lambda_0$ for which $\supp(\hat{\beta}_\lambda)=\supp(\beta^0)$. A necessary condition for this to be meaningful is of course that the target is linear, i.e. $f^0(x)={\beta^0}'x$ for some $\beta^0\in\mathbb{R}^p$, but nothing prevents one from using GIC even when the target is non-linear. The theoretical merits of the procedure are, to our knowledge, unknown in that case. 

At this point it is also worth mentioning that the quadratic loss is covered by (\ref{GIC}) since
\begin{align*}
\frac{1}{2}\del[1]{y-\beta'x}^2=\frac{1}{2}y^2+\frac{1}{2}(\beta'x)^2-y(\beta'x)
\end{align*}
such that we may choose $b(u)=\frac{1}{2}u^2$ and $c(y,\xi)=-\frac{1}{2}y^2$. Similarly, the logistic loss is covered since 
\begin{align*}
\rho_f(y,x)=-y(\beta'x)+\log(1+e^{\beta'x})
\end{align*} 
such that one may choose $b(u)=\log(1+e^u)$ and $c(\cdot, \cdot)=0$. This covers a large class of convex functions as in \cite{vdG08}. Regarding the choice of threshold for variable selection we suggest using the procedure outlined in \cite{canerk13}. 

\section{Econometric Examples}\label{Examples}
In this section we present sufficient conditions for Assumptions 1-3 to be valid for concrete econometric examples. Recall that these Assumptions are sufficient for Theorem \ref{thm1} to be valid. When necessary, we also comment on sufficient conditions for the assumptions underlying Theorem \ref{pbound}. As already argued in connection with this theorem the critical assumption is the Lipschitz continuity of the loss function while the boundedness assumption on the basis functions is rather innocuous. As a consequence, we will focus on verification of the Lipschitz continuity whenever this is not trivially satisfied. First, we present a general sufficient condition for loss functions to satisfy a quadratic margin condition. This condition is then exemplified on a couple of examples. In the following $\enVert[0]{\cdot}$ denotes the $L_2(P)$-norm on $\mathbf{F}$ and the data is supposed to come from an i.i.d. sample.

Assume that the loss function is of the form $\rho_f(x,y)=\rho(f(x),y)$ such that it only depends on $x$ through $f(x)$. By Doob's representation, see Lemma 1.13 in \cite{kallenberg02}, we define
\begin{align*}
l(f(X),X):=E[\rho(f(X),Y)|(X, f(X))]
\end{align*}
Furthermore, by iterated expectations, it suffices to show that $l(f(x),x)$ satisfies the margin condition in order to verify that this is the case for $\rho_f(x,y)=\rho(f(x),y)$. The target function will be 
\begin{align*}
f^0(x)=\argmin_{f\in\mathbf{F}} l(f(x),x)
\end{align*}
which is the minimizer of the loss function and hence a natural choice. Next, fix an $x\in \mathcal{X}$ and assume that $l(a,x)$ is twice continuously differentiable in its first argument in an $\enVert{\cdot}_\infty$-neighborhood of radius $\eta>0$ around $f^0(x)$ with second derivative bounded from below by $2c>0$, for $c>0$. Then it follows by Lagrange's form of the remainder term in a Taylor series that for some $\tilde{a}$ on the line segment joining $f(x)$ and $f^0(x)$ 
\begin{align*}
l(f(x),x)
&=
l(f^0(x),x)+l_1'(f^0(x),x)(f(x)-f^0(x))+\frac{l_{11}''(\tilde{a},x)}{2}(f(x)-f^0(x))^2\\
&\geq
l(f^0(x),x)+c(f(x)-f^0(x))^2
\end{align*}
for all $f\in \mathbf{F}$ such that $|f(x)-f^0(x)|<\eta$, and $l_1'(\cdot, \cdot),\ l_{11}''(\cdot, \cdot)$ represent the first and second order partial derivatives of $l$ with respect to its first argument. Assuming this is valid for all $x\in\mathcal{X}$ implies that %
\begin{align}
l(f(x),x)-l(f^0(x),x)
\geq
c(f(x)-f^0(x))^2\label{trickaux}
\end{align}
for all $f\in \mathbf{F}_{\text{local}}=\cbr[0]{f\in\mathbf{F}:\enVert[0]{f-f^0}_{\infty}\leq\eta}$. This yields, using that the $(X_i,Y_i)_{i=1}^n$ are identically distributed,
\begin{align*}
\Xi(f)
=
P\rho_f-P\rho_{f^0}
=
\sbr[1]{E\rho_f(X_1,Y_1)-E\rho_{f^0}(X_1,Y_1)}
=
E\sbr[1]{l(f(X_1,X_1)-l(f^0(X_1),X_1)}
\geq
c\enVert[0]{f-f^0}^2
\end{align*}
for all $f\in\mathbf{F}_{\text{local}}$ such that the margin condition is satisfied with $G(x)=cx^2$. Put differently, the above shows that it suffices to establish a lower bound on the second derivative of the conditional expectation of the loss function in order to show that the margin condition holds with quadratic margin. 

We shall next use the above result to verify that some typical loss functions encountered in econometrics satisfy Assumptions 1-3. 

\subsection{Quadratic loss}\label{subsecql}
Assume that the data is generated by the i.i.d. sequence
\begin{align}
Y_i=f^0(X_i)+\epsilon_i\label{npreg}
\end{align}
for $X_i\in\mathcal{X}\subseteq \mathbb{R}$.\footnote{It is not difficult to generalize this to $\mathcal{X}$ being some subset of a normed space by modifying the definition of sub-gaussianity in footnote \ref{sg} below slightly.} We show that the quadratic loss function 
\begin{align}
\rho(f(x),y)=\del[0]{y-f(x)}^2\label{qloss}
\end{align}
for $f\in\mathbf{F}$ can be encompassed by our general theory. The quadratic loss function is probably the most widely used loss in regression analysis. The main obstacle in fitting this type of loss into our general theory is that Theorems \ref{pbound} and \ref{thm3} rely on the loss function being Lipschitz continuous in order to lower bound the probability of the event $\tau$. However, the quadratic loss is only locally Lipschitz continuous. As we shall see, this can still deliver an oracle inequality which holds with high probability if the covariates and the error terms don't have too heavy tails and the target function $f^0$ is bounded on compact subsets of $\mathbb{R}$. Before stating the first result note that  for any $f\in \mathbf{F}$, if $\epsilon_1$ is independent of $X_1$
\begin{align}
\Xi(f)=E(Y_1-f(X_1))^2-E(Y_1-f^0(X_1))^2=E(f(X_1)-f^0(X_1))^2\label{qm}
\end{align}
such that the excess loss reduces to the mean square error (MSE) in case of a quadratic loss function. This turns out to be particularly useful when discussing the relation of our procedure to nonparametric series estimation in Section \ref{npsec} below. (\ref{qm}) also reveals that the margin condition is satisfied with a quadratic margin, even without using the technique from above, which implies that Assumption 1 is satisfied. 

\textbf{Remark:} If the sample is merely assumed to be independently but not necessarily identically distributed the above calculations still go through with some small modifications. We explain how next. First, let the norm on $\mathbf{F}$ be $\enVert[0]{f}=\sqrt{\sum_{i=1}^n\frac{1}{n}Ef^2(X_i)}$. Note that this reduces to the norm from the identically distributed case in case the variables are actually identically distributed. Assuming that $\epsilon_i$ is independent of $X_i$ for all $i=1,...,n$ one gets by the same arguments as above
\begin{align*}
\Xi(f)=\frac{1}{n}\sum_{i=1}^nE(Y_i-f(X_i))^2-\frac{1}{n}\sum_{i=1}^nE(Y_i-f^0(X_i))^2=\frac{1}{n}\sum_{i=1}^nE(f(X_i)-f^0(X_i))^2
\end{align*}
which shows that the margin condition is satisfied with a quadratic margin function even in the case where the sample is not identically distributed. In particular, this means that the bounds to follow, including the nonparametric ones in Section \ref{npsec}, are valid in the presence of heteroscedasticity. For the sake of exposition we shall now return to the i.i.d. situation. 

\begin{lemma}\label{quad}
Assume that $\max_{1\leq j\leq p}\enVert[0]{\psi_j}_\infty \leq K$. Set $\lambda_2=\frac{\lambda_1\sqrt{s^*}}{2 \enVert{\beta^*}_2}$ and take for some constant $L\geq 8$, $8 \lambda_0 \le \lambda_1 \le L \lambda_0$ and suppose that
\begin{align}
\frac{36 L^2 \lambda_0 s^*K}{\eta \phi_*^2} \leq 1\label{quadinrec}.
\end{align}
In addition, assume that $\|f^* - f^0 \|_{\infty} \leq \eta/2$, with $\frac{3}{2}E(f^*-f^0)^2 \leq \frac{9 \lambda_1^2s^*}{\phi_*^2}$ and finally that the population covariance matrix of the covariates $\Sigma=E(\psi(X)\psi'(X))$ has full rank. Then,

a) Assumptions 1-3 are valid and on the set $\tau$
\begin{align}
E(\hat{f}-f^0)^2 + \lambda_1 \| \hat{\beta} - \beta^* \|_1 + \lambda_2 \|\hat{\beta}_{S^*} - \beta_{S^*}\|_2^2
\leq
6 E(f^*-f^0)^2+  \frac{36 L^2\lambda_0^2 s^*}{\phi_*^2}\label{quada}.
\end{align}

b) if, furthermore $\frac{1}{n} \sum_{i=1}^n \max_{ 1 \le j \le p} E \psi_j^2 (X_i) \leq 1$, $\sup_{|x|\leq C_n}|f^0(x)|\leq F_{C_n}<\infty$ for all $C_n>0$, $\Phi=\cbr[0]{\beta\in\mathbb{R}^p:\enVert[0]{\beta}_1\leq G<\infty}$ and $X_1,\epsilon_1$ are sub-gaussian\footnote{A real random variable $V$ is said to be sub-gaussian if $P(|V|\geq x)\leq \alpha\exp(-\delta x^2)$ for some positive constants $\alpha$ and $\delta$.\label{sg}} one has for all $C_n>0$ that (\ref{quada}) is valid with probability at least $1-\del[2]{\frac{1}{p}}-2\alpha n\exp(-\delta C_n^2)$ for $\lambda_0=dD_n\sqrt{\frac{\log(p)}{n}}$ with $D_n=2(C_n+2F_{C_n}+GK)$ and $d>0$.
\end{lemma}

Lemma \ref{quad} a) provides sufficient conditions for Assumptions 1-3, and a bound similar to the one in Theorem \ref{thm1} which is valid on the set $\tau$. Lemma \ref{quad} b) also provides a lower bound on the probability of the set $\tau$ in the case of a non-Lipschitz continuous loss function. Some of the assumptions of Lemma \ref{quad} may still seem rather high level but they are not very restrictive. We next give a concrete example of when they are satisfied. In particular, we shall assume that $x$ has support in $[-1,1]$. This can of course be achieved by a suitable (sigmoidal) transformation of the covariates. $f^0$ will be assumed to belong to a  H\"{o}lder class of order $1/2<r<\infty$ (or it is $r$-smooth in the terminology of \cite{chen07} \footnote{Following \cite{chen07}, we say that a function $f$ is $r$-smooth if it has derivatives up to order $\lfloor r\rfloor$ and the $r$-th derivative is H\"{o}lder continuous with an exponent of $r-\lfloor r\rfloor$. Here $\lfloor r\rfloor$ denotes the greatest integer strictly less than $r$. This is also often called a H\"{o}lder class of order $r$, see eg \cite{Tsybakov09}.\label{SH}}). In this case one may choose $\mathcal{F}_L$ to consist of $p$th degree polynomials, i.e. $\psi_j(x)=x^j,\ j=1,...,p$, and $\Gamma=\cbr[0]{1,...,s^*}$. Using other basis functions than polynomials is of course a possibility if one does not want to assume that the covariates are contained in $[-1,1]$. However, this is a rather innocent assumption since it can be obtained by transforming the covariates. Alternatively, one may use a basis of bounded functions. In any case, it is not our purpose to promote one basis over another. We simply show that polynomials suffice for our theory to work.

The assumption $\max_{1\leq j\leq p}\enVert[0]{\psi_j}_\infty \leq K$ requires the basis functions $\psi_j$ to be bounded. Since the covariates have support in $[-1,1]$ it follows that $\enVert[0]{\psi_j}_\infty=\enVert[0]{x^j}_\infty=1$ for all $j=1,...,p$ and hence the assumption is satisfied with $K=1$. As mentioned right after (\ref{4.3}) this also implies $\frac{1}{n} \sum_{i=1}^n \max_{ 1 \le j \le p} E \psi_j^2 (X_i) \leq 1$ which is needed in part b) of Lemma \ref{quad}.
 
The approximation requirements $\|f^* - f^0 \|_{\infty} \le \eta/2$ and $\frac{3}{2}E(f^*-f^0)^2 \leq \frac{9 \lambda_1^2s^*}{\phi_*^2}$ state that the target should be approximated well by the oracle and imply that $f^*\in \mathbf{F}_{\text{local}}$. As explained previously, one may choose $f^*=f^0$ when $f^0$ is linear. Hence, the two approximation requirements are trivially satisfied in the linear case. In general the validity of these high level assumptions depend on the choice of $\mathbf{F}$ and the existence of properly approximating basis functions $\cbr[0]{\psi_j}_{j=1}^\infty$. As mentioned, we shall assume that the covariates are compactly supported. In general, assuming that $f^0$ belongs to a H\"{o}lder class of order $1/2<r<\infty$, choosing $\mathcal{F}_L$ to consist of $p$th degree polynomials, i.e. $\psi_j(x)=x^j,\ j=1,...,p$, and $\Gamma=\cbr[0]{1,...,s^*}$ it follows from page 5573 in \cite{chen07} that $\|f^* - f^0 \|_{\infty}\in O({s^*}^{-r})$. Hence, $\|f^* - f^0 \|_{\infty}\leq \eta/2$ will be satisfied for any $\eta>0$ as long as $s$ is sufficiently large. Furthermore, since
\begin{align}
E(f^*-f^0)^2
=
\enVert[0]{f^*-f^0}^2
\leq
\enVert[0]{f^*-f^0}_{\infty}^2\in O({s^*}^{-2r})\label{scond}
\end{align} 
one also has that $\frac{3}{2}E(f^*-f^0)^2\in O({s^*}^{-2r})$ such that  $\frac{3}{2}E(f^*-f^0)^2\leq \frac{9 \lambda_1^2s^*}{\phi_*^2}$ if $s^*$ is at least of the order $\lambda_1^{-2/(2r+1)}$, i.e. $s^*\in \Omega\del[1]{\lambda_1^{-2/(2r+1)}}, $ and $\phi_*$ is assumed to be bounded away from 0.

The condition (\ref{quadinrec}) is not restrictive either. Assuming that $\phi_*$ is bounded away from 0 it suffices to show that $\lambda_0s^*\to 0$. Choosing $s^*\in\Theta\del[1]{\lambda_1^{-2/(2r+1)}}$, which is in accordance with $s^*\in \Omega\del[1]{\lambda_1^{-2/(2r+1)}}=\Omega\del[1]{\lambda_0^{-2/(2r+1)}}$ (since $\lambda_1$ and $\lambda_0$ are of the same order), yields
\begin{align}
\lambda_0s^*\in \Theta\del[2]{\lambda_0\lambda_0^{-2/(2r+1)}}=\Theta\del[2]{\lambda_0^{(2r-1)/2r}}\label{thetabound}.
\end{align} 
Hence, if $r>1/2$, it is enough that $\lambda_0\to 0$. Since the covariates are assumed to have compact support and $f^0$ is more than $\frac{1}{2}$-smooth, and hence continuous, one has $F_{C_n}$ is bounded and so $\lambda_0\in \Theta\del[1]{D_n\sqrt{\frac{\log(p)}{n}}}=\Theta\del[1]{C_n\sqrt{\frac{\log(p)}{n}}}$ (as seen next we think of $C_n$ as an increasing sequence). For asymptotic considerations an obvious choice of $C_n$ is $C_n=\sqrt{\frac{2}{\delta}\log(n)}$ since this ensures that $2\alpha n\exp(-\delta C_n^2)$ tends to zero in the lower bound on the probability with which (\ref{quada}) holds in part b) of Lemma \ref{quad}. With this choice, $\lambda_0\in \Theta\del[1]{\sqrt{\log(n)}\sqrt{\frac{\log(p)}{n}}}$ and it suffices that $\sqrt{\log(n)}\sqrt{\frac{\log(p)}{n}}\to 0$ which still permits $p$ to increase at an almost exponential rate in $n$. Note also at this point that since $s^*\in\Theta\del[1]{\lambda_1^{-2/(2r+1)}}$, $\lambda_0\in \Theta\del[1]{\sqrt{\log(n)}\sqrt{\frac{\log(p)}{n}}}$ implies $s^*\in \Theta\del[1]{\sbr[0]{\frac{n}{\log(p)\log(n)}}^{1/(2r+1)}}$. So the more smooth ($r$ large) the function $f^0$ is the smaller can we afford to choose $s^*$ since smooth functions are easier to approximate. This is of course desirable since this means few coefficients have to be estimated which results in a smaller estimation error.    

Assuming $\Sigma$ to have full rank is reasonably innocent as discussed in connection with the adaptive restricted eigenvalue condition in Section \ref{setup}.

The additional assumptions of part b) of Lemma \ref{quad} are used to establish a lower bound on the probability of $\tau$ on which (\ref{quada}) is valid in the absence of Lipschitz continuity of the loss function. First, $\frac{1}{n} \sum_{i=1}^n \max_{ 1 \le j \le p} E \psi_j^2 (X_i) \leq 1$, imposes a further boundedness assumption on the basis functions, which, as discussed above, is trivially satisfied since $K=1$. The assumption $\sup_{|x|\leq C_n}|f^0(x)|\leq F_{C_n}<\infty$ requires $f^0$ to be locally bounded and is satisfied in particular if $f^0$ is continuous. Continuity of $f^0$ is of course ensured if the target is linear or $r$-smooth for $r>0$ (the latter being covered by our working example). Hence, our theory covers the case of a linear target with quadratic loss. If, as in the discussion above, we continue to assume that the covariates are compactly supported $F_{C_n}$ can even be chosen to be an absolute constant independent of the sample size (as already mentioned just after (\ref{thetabound})). The assumption $\enVert{\beta}_1\leq G$ is rather innocent since $G$ can be chosen arbitrarily large\footnote{It is also straightforward to generalize to the setting where $G$ is a sequence depending on the sample size.}. The sub-gaussianity assumption on the covariates and the error terms prevents them from having too heavy tails. This is quite common assumption in the literature on high-dimensional models. When the covariates are bounded they are, a fortiori, subgaussian. In conclusion, all conditions of Lemma \ref{quad} are valid when $f^0$ belongs to a H\"{o}lder class of order at least $1/2$ and the covariates have support $[-1,1]$. In particular, choosing $\mathcal{F}_L$ to consist of $p$th degree polynomials, i.e. $\psi_j(x)=x^j,\ j=1,...,p$, and $\Gamma=\cbr[0]{1,...,s^*}$ is sufficient.

\textbf{Remark:} Even though the sub-gaussianity assumptions on the covariates and the error terms is a standard one we would like to stress that our theory is also applicable for much heavier tails. A version of Lemma \ref{quad} b) can be developed even when we only have $E|X_1|^r,\ E|\epsilon_1|^r\leq\kappa<\infty$ for some $r\geq 2$. As long as the covariates and error terms posses enough moments Lemma \ref{quad} b) still applies. More precisely, a slight change of the last part of the proof of part b) yields that $1-\del[2]{\frac{1}{p}}-\frac{\kappa n}{C_n^r}$. So the price to pay for the increased generality is that the last term in the lower bound on the probability no longer tends to 0 exponentially fast in $C_n$ such that $C_n$ now has to be at least of the order $O(n^{1/r})$.

\subsubsection{Relation to non-parametric series estimators}\label{npsec}
At this point it is worth pointing out that our theory encompasses non-parametric series estimation as a special case since $\hat{f}(x)=\sum_{j=1}^{p}\hat{\beta}_j\psi_j(x)$ can be seen as a series estimator of the unknown function $f^0(x)$. Note first that our results are mainly non-asymptotic  while we are not aware of any finite sample results for classical series estimators. Apart from this, our estimator has another big advantage compared to the classical series estimator of e.g. \cite{newey97}: In conventional series estimators of the form $\sum_{j=1}^{T}\hat{\beta}_j\psi_j(x)$ the number of terms $T$ in the series can usually only increase slower than the sample size such as $T\in o(n^{1/2})$ or $T\in o(n^{1/3})$ in \cite{newey97}. However, as we have already hinted at several times, the number of series terms $p$ in our estimator can increase exponentially fast in asymptotic considerations. We shall be more precise about this in the discussion of Corollary \ref{quadcor} below. What we require instead is that the number of \textit{non-zero} terms, $s$, in our series estimator increases slowly. As remarked already, though in a slightly different context, by \cite{bellonic11} this is still a considerable generalization since we do not require that it is the \textit{first} $s$ terms in $f^*(x)=\sum_{j=1}^{p}\beta^*_j\psi_j(x)$ that have the non-zero coefficients. Put differently, we do not assume any prior knowledge of the location of the non-zero coefficients of the oracle estimator. To be concrete, assume that $\mathbf{F}=L_2(P_{X_1})$ with the usual inner product $\langle\cdot,\cdot\rangle$ and $P_{X_1}$ being the distribution of $X_1$. This is a Hilbert space and hence there exists an orthonormal basis $\cbr[0]{g_j}_{j=1}^\infty$ such that every $f\in\mathbf{F}$ can be written as
\begin{align*}
f(x)=\sum_{j=1}^\infty\langle f, g_j\rangle g_j(x)
\end{align*} 
Using the first $s$ terms of this series representation as an approximation of $f$ yields the approximation $\tilde{f}(x)=\sum_{j=1}^s\langle f, g_j\rangle g_j(x)$ with corresponding approximation error $\sum_{j=s+1}^\infty\langle f, g_j\rangle^2$. More generally, letting $\mathcal{I}_s$ be any subset of cardinality $s$ of the natural numbers, one could instead consider the approximation $\breve{f}(x)=\sum_{j\in\mathcal{I}_s}\langle f, g_j\rangle g_j(x)$ with approximation error $\sum_{j\in\mathcal{I}_s^c}\langle f, g_j\rangle^2$. Clearly, choosing $\mathcal{I}_s$ to consist of the $s$ largest $\langle f, g_j\rangle^2$ yields the smallest approximation error. However, there is no reason to believe that the $s$ largest terms in $\cbr[0]{\langle f, g_j\rangle^2}_{j=1}^\infty$ are also the $s$ first ones, in fact their location is generally unknown and hence our method is useful since it allows for $p>>s$ candidate terms. Put differently, we choose $\mathcal{I}_s$ from $\cbr[1]{A\subset\cbr[0]{1,...,p}: |A|=s}$ and clearly $\cbr[0]{1,...,s}\in \cbr[1]{A\subset\cbr[0]{1,...,p}: |A|=s}$. This is considerably more general than simply choosing $\mathcal{I}_s=\cbr[0]{1,...,s}$ since $s<<p$.

Note also that \cite{canerz13} have shown that the choice of dimension matters in structural function estimation via sieves. As we shall see now, as long as $f^0$ can be "well approximated" by some sparse linear combination $f^*(x)=\sum_{j=1}^{p}\beta^*_j\psi_j(x)$, part b) of Lemma \ref{quad} will provide a \textit{finite} sample upper bound on the estimation error of $\hat{f}$ (in an $L_2$-sense). In particular, part b) of Lemma \ref{quad} yields
\begin{corollary}\label{nonpar}
a) Let the assumptions of Lemma \ref{quad} be satisfied. Then
\begin{align}
E(\hat{f}-f^0)^2
\leq
6E(f^*-f^0)^2 +  \frac{36 L^2\lambda_0^2 s^*}{\phi_*^2}\label{seriesest}.
\end{align}
with probability at least $1-\del[2]{\frac{1}{p}}-2\alpha n\exp(-\delta C_n^2)$. 

b)
Assume furthermore that the covariates have support $[-1,1]$, $f^0$ is $r$-smooth for $1/2<r<\infty$ (belongs to a H\"{o}lder class of order $r$ as defined in footnote \ref{SH}) and we choose $\mathcal{F}_L$ to consist of polynomials of degree $s^*$. Then if $\Phi$ is compact, $\epsilon_1$ sub-gaussian and $\phi_*$ bounded away from 0.
\begin{align}
E(\hat{f}-f^0)^2
\leq
O\del[1]{{s^*}^{-2r}} +  \frac{36 L^2\lambda_0^2 s^*}{\phi_*^2}\label{seriesest2}.
\end{align}
with probability at least $1-\del[2]{\frac{1}{p}}-2\alpha n\exp(-\delta C_n^2)$.

c) Maintain the assumptions of b). Choosing $s^*\in\Theta\del[2]{\lambda_0^{\frac{-2}{2r+1}}}$ yields
\begin{align}
E(\hat{f}-f^0)^2
\leq O\del[2]{\lambda_0^{\frac{4r}{2r+1}}}=O\del[3]{C_n^{\frac{4r}{2r+1}}\sbr[2]{\frac{\log(p)}{n}}^{2r/(2r+1)}}\label{seriesest3}
\end{align}
with probability at least $1-\del[2]{\frac{1}{p}}-2\alpha n\exp(-\delta C_n^2)$.

d) Finally, if in addition one chooses $C_n=\sqrt{\frac{2}{\delta}\log(n)}$ \footnote{Where $\delta>0$ is the subgaussianity constant of footnote \ref{sg}.}, one has
\begin{align}
E(\hat{f}-f^0)^2\in O\del[3]{\log(n)^{\frac{2r}{2r+1}}\sbr[2]{\frac{\log(p)}{n}}^{2r/(2r+1)}}\label{seriesest4}
\end{align}
with probability at least $1-\frac{1}{p}-\frac{2\alpha}{n}$.
\end{corollary}

(\ref{seriesest}) of Corollary \ref{nonpar} is remarkable since it provides a finite sample upper bound on the mean square error (MSE) of the Elastic Net series estimator $\hat{f}$ of $f^0$. The upper bound on the MSE depends on two terms. First, $\Xi (f^*)$, the approximation error stemming from approximating $f^0$ by $f^*$. The second term, $\frac{36 L^2\lambda_0^2 s^*}{\phi_*^2}$, can be interpreted as the estimation error of $\hat{\beta}$ in estimating the oracle parameter $\beta^*$. The larger the number of coefficients to be estimated, the larger will the estimation error be, but the approximation error will be smaller. In (\ref{seriesest2}) we restrict attention to bounded covariates (which are, a fortiori, sub-gaussian) and smooth functions. This allows us to give an explicit order of magnitude on $E(f^*-f^0)^2$ as explained in the discussion after Lemma \ref{quad}. Here it is worth mentioning that the order of approximation error is based on results of approximation by polynomials of degree $s^*$. However, our discussion prior to Corollary \ref{nonpar} indicates that it is not necessarily a linear combination of $x,x^2,...,x^{s^*}$ which yields the best approximation and so it might be possible to further improve the bound in (\ref{seriesest2}). But this is a topic in approximation theory which we do not feel is appropriate to pursue here. However, our estimator may perform perform much better in practice than the plain series estimator. 

(\ref{seriesest3}) follows from (\ref{seriesest2}) by choosing $s^*$ to minimize the order of the upper bound in (\ref{seriesest2}). We remark here that this is exactly in accordance with the choice of $s^*$ in connection with (\ref{scond}) and (\ref{thetabound}) which ensured the validity of the approximability conditions in Lemma \ref{quad} since $\lambda_1$ and $\lambda_0$ are of the same order. Note that (\ref{seriesest4}) gives an explicit order of the estimation error in terms of $p$ and $n$ only. It reveals that the $L_2$-distance between $\hat{f}$ and $f^0$ tends to zero even when $p$ is allowed to increase almost exponentially fast in $n$. We also remark that by utilizing the boundedness of the covariates in the proof of part b) of Lemma \ref{quad} the lower bound on the probability with which (\ref{seriesest2}) and (\ref{seriesest3}) are valid may be increased to $1-\del[2]{\frac{1}{p}}-\alpha n\exp(-\delta C_n^2)$. This can further be increased to $1-\frac{1}{p}$ if the error terms are bounded as well. The corresponding lower bounds on the probability with which (\ref{seriesest4}) is valid may similarly be increased to $1-\frac{1}{p}-\frac{\alpha}{n}$ and $1-\frac{1}{p}$, respectively. 

\subsubsection{Asymptotic results for quadratic loss}
We now return to the general setting of Lemma \ref{quad}. To illustrate the usefulness of Lemma \ref{quad} we remark that it can be used to establish the following asymptotic result. As in Corollary \ref{corlinasym} we assume that $p\in O(\exp(n^a))$ and $s^*\in O(n^b)$ for some $a>0$ and $b\geq 0$. Furthermore, we assume that $F_{C_n}\in O(n^{\tilde{d}})$ for some $\tilde{d}\geq 0$.

\begin{corollary}\label{quadcor}
Let the assumptions of Lemma \ref{quad} be satisfied. Choose $C_n=\sqrt{\frac{2}{\delta}\log(n)}$ \footnote{Where $\delta>0$ is the subgaussianity constant of footnote \ref{sg}.}. Then, if $\phi_*^2$ is bounded away from 0, one has with probability tending to one,

a) 
\begin{align*}
\limsup_{n\to\infty} E(\hat{f}-f^0)^2
\leq
6\limsup_{n\to\infty} E(f^*-f^0)^2
\end{align*}
if $a+b+2\tilde{d}<1$.

b) 
\begin{align*}
|\hat{\beta}-\beta^0|=\| \hat{\beta} - \beta^0 \|_1\to 0
\end{align*}
if the target $f^0$ is linear and $a+2\tilde{d}<1$.
\end{corollary}
The assumption $F_{C_n}\in O(n^{\tilde{d}})$ is not overly restrictive since even when $f^0(x)=\exp(\mu x^2)$ for some $\mu>0$ one has that $F_{C_n}=\sup_{|x|\leq C_n}|f^0(x)|=\exp(2\frac{\mu}{\delta}\log(n))=n^{2\frac{\mu}{\delta}}$. So the assumption of a polynomial growth of $F_{C_n}$ can be satisfied even by functions increasing exponentially fast by choosing $\tilde{d}=2\frac{\mu}{\delta}$. In the case where we only assume that the covariates and the error terms have bounded $r$th moments a similar argument shows that $f^0$ can increase at the rate of an $r$th degree polynomial. We have assumed in Corollary \ref{quadcor} that the assumptions of Lemma \ref{quad} are valid. Recall that we know this is the case if $f^0$ belongs to a H\"{o}lder class of order $r>1/2$ and the covariates have support $[-1,1]$. Then we may choose $\mathcal{F}_L$ to consist of polynomials of degree $p$ and $\Gamma=\cbr[0]{1,...,s}$ to meet the assumptions of Lemma \ref{quad}. 

Part a) of Corollary \ref{quadcor} is similar to Corollary \ref{corloss} but can not be deduced from it since the conditions of the latter Corollary are not satisfied in the case of quadratic loss. Similarly, part b) of Corollary \ref{quadcor} resembles Corollary \ref{corlinasym} but can not be deduced from it since the conditions of the latter corollary are not satisfied in the case of quadratic loss. Recall as well that we also assume throughout Section \ref{subsecql} that $x$ is one-dimensional which explains the equality in part b) above.

\subsection{Logistic regression}
Next we verify that the logistic loss satisfies Assumptions 1-3.  This is done by verifying that the second derivative with respect to the first argument of the conditional loss function is bounded away from zero as discussed in the beginning of Section \ref{Examples}. Let $(Y_i, X_i)_{i=1}^n$ be i.i.d. As seen in Section \ref{setup} the loss in case of logistic regression is
\begin{equation}
\rho(f(x),y) = - y f(x) + \log (1+ \exp(f(x)).\label{5.1}
\end{equation}
Here $y,x$ represent the values of the random variable $Y_i$ and $X_i$.  Then we have the following result.

\begin{lemma}\label{logit}
Assume that $ \epsilon_0 \leq \pi(x) \leq 1 - \epsilon_0, \ \text{for all } x\in\mathcal{X}$
for some  $0 < \epsilon_0 \leq 1/2$,
and $\max_{1\leq j\leq p}\enVert[0]{\psi_j}_\infty \le K$. Set $\lambda_2=\frac{\lambda_1\sqrt{s^*}}{2 \enVert{\beta^*}_2}$ and take for some constant $L\geq 8$, $8 \lambda_0 \le \lambda_1 \le L \lambda_0$
and suppose that
\[ \frac{36 L^2 \lambda_0 s^*K}{\eta \phi_*^2c} \le 1.\]
for $c=\frac{\epsilon_0e^{-\eta}}{2\del[1]{1+e^\eta/\epsilon_0}^2}$. Assume that $\|f^* - f^0 \|_{\infty} \le \eta/2$, with $\frac{3}{2}\Xi (f^*) \leq \frac{9 \lambda_1^2s^*}{\phi_*^2 c}$ and finally that the population covariance matrix of the covariates $\Sigma=E(\psi(X)\psi'(X))$ has full rank. Then,

a) Assumptions 1-3 are valid and on the set $\tau$,

\begin{align}
\Xi (\hat{f}) + \lambda_1 \| \hat{\beta} - \beta^* \|_1 + \lambda_2 \|\hat{\beta}_{S^*} - \beta_{S^*}\|_2^2
\leq
6 \Xi (f^*) +  \frac{ 36  \lambda_1^2 s^*}{\phi_*^2 c}.\label{logitbound1}
\end{align}

b) if, furthermore, $ \frac{1}{n} \sum_{i=1}^n \max_{ 1 \le j \le p} E \psi_j^2 (X_i) \le 1$ one has that (\ref{logitbound1}) is valid with probability at least $1-\del[2]{\frac{1}{p}}$ upon choosing $\lambda_0=2d\sqrt{\frac{\log(p)}{n}}$ (where $d$ is the positive constant from Theorem \ref{pbound}).
\end{lemma}

Note that except for $\epsilon_0 \leq \pi(x) \leq 1 - \epsilon_0$ for all $x\in\mathcal{X}$ the assumptions of Lemma \ref{logit} are a subset of the ones in Lemma \ref{quad}. Hence, we know that they are satisfied when the covariates have support $[-1,1]$ and $f^0$ belongs to a H\"{o}lder class of order $r>1/2$ since then we may choose $\psi_j(x)=x^j,\ j=1,...,p$ and $\Gamma=\cbr[0]{1,...,s}$ as discussed after Lemma \ref{quad}. 

(\ref{logitbound1}) in Lemma \ref{logit} gives upper bounds on the excess risk as well as the estimation error in the logit model. Corollaries \ref{corloss}-\ref{corlinasym} can now be used to establish asymptotic results for these quantities.

By the definition of $\lambda_0$ it is not difficult to see that the second term on the right hand side tends to zero if $s^*\in o(n/\log(p))$ and $\phi_*$ is bounded away from zero. This in turn reveals that $\Xi(\hat{f})$ is of the same order of magnitude as the excess risk of the oracle $\Xi (f^*)$ which parallels part a) in Corollary \ref{quadcor}. Mimicking the arguments in part b) of that Corollary one sees that 
\begin{align*}
\| \hat{\beta} - \beta^0 \|_1\to 0
\end{align*}
with probability tending to one if $\phi_*^2$ is bounded away from 0 and the target $f^0$ is linear as long as $s^*\log(p)/n\to 0$.

\section{Conclusion}
This paper has established an oracle inequality for empirical loss minimization penalized by the elastic net penalty. This inequality is valid for convex loss functions and non-linear target functions and we stress again that this is a finite sample result. We have also seen that the results for a Lasso penalty are a special case of ours. For the case where the target is linear the oracle inequality can be used to deduce finite sample upper bounds on the estimation error of $\hat{\beta}$. The oracle inequality can also be used to show that the excess loss of our estimator is asymptotically of the same order as that of the oracle. Also, when the target is linear we give sufficient conditions for $\hat{\beta}$ to be consistent for $\beta^0$ 

Furthermore, we explain how to construct a thresholded elastic net estimator which can perform consistent variable selection when the truth is linear. To illustrate the generality of our framework we give two examples of settings which fit into our theory --  the quadratic and the logistic loss. In the case of a quadratic loss function we allow for heteroscedastic error terms. In addition, we show how our results provide new insights into nonparametric series estimation which turns out to be encompassed as a special case of our theory. More precisely, we provide an upper bound on the mean square error of the elastic net series estimator.

Future avenues of research include, but are not limited to, proposing theoretically justified data driven methods for the choice of tuning parameters in the case where the truth is not linear. Also, bounds which are valid for dependent data are interesting for time series analysts. Finally, extending our results to allow for heteroscedasticity in other loss functions than the quadratic one is of interest for practical purposes.

\section*{Appendix}\label{appendix}
We start by proving Theorem \ref{thm1}. As a by-product of the proof we extend the basic inequality (6.28) of \cite{vdGB11} to cover the elastic net using a refined proof technique.
\begin{proof}[Proof of Theorem \ref{thm1}]
The proof consists of three steps. In the first step we set up a basic inequality which we shall use in steps two and three. In the second and third steps we analyze two possibilities:  the penalties multiplied by estimation errors being greater than or equal to the oracle rate $\Delta^*$, or these penalties multiplied by estimation errors being strictly less than the oracle rate, respectively. Throughout the appendix we shall let $S_c$ denote the complement of the set $S$ for any set $S$ \footnote{It will be clear from the context of which set $S$ is a subset.}.

{\bf Step 1}. This step uses the convexity of the loss function $\rho_f$. First, define
\[ \tilde{\beta} = t \hat{\beta} + (1-t) \beta^*,\]
where
\[ t = \frac{M^*}{M^* + \| \hat{\beta} - \beta^* \|_1}.\]
To simplify notation define $\tilde{\Xi} = \Xi (f_{\tilde{\beta}})$ and $ \Xi^* = \Xi (f_{\beta^*})$ 
and set $\tilde{f}= f_{\tilde{\beta}}$, $\hat{f} = f_{\hat{\beta}}$, $f^* = f_{\beta^*}$. By the minimizing property of $\hat{\beta}$ one has that
\begin{align}
P_n\rho_{\hat{f}}+\lambda_1\|\hat{\beta}\|_1+\lambda_2\|\hat{\beta}\|_2^2
\leq
P_n\rho_{f^*}+\lambda_1\|\beta^*\|_1+\lambda_2\|\beta^*\|_2^2 \label{iqaux}
\end{align}
By the convexity of $\beta \mapsto \rho_\beta$  and the linearity of the $P_n$-integral it follows that
\begin{align}
P_n\rho_{\tilde{f}}+\lambda_1\|\tilde{\beta}\|_1+\lambda_2\|\tilde{\beta}\|_2^2
&\leq
tP_n\rho_{\hat{f}}+(1-t)P_n\rho_{f^*}+t\lambda_1\|\hat{\beta}\|_1+(1-t)\lambda_1\|\beta^*\|_1+t\lambda_2\|\hat{\beta}\|_2^2+(1-t)\lambda_2\|\beta^*\|_2^2\notag \\
&\leq 
P_n\rho_{f^*}+\lambda_1\|\beta^*\|_1+\lambda_2\|\beta^*\|_2^2 \label{conv}
\end{align}
where the second inequality follows from (\ref{iqaux}). Rearranging (\ref{conv}) yields
\begin{equation}
[ P_n \rho_{\tilde{f}} - P_n \rho_{f^*} ] + \lambda_1 \| \tilde{\beta} \|_1 + \lambda_2 \| \tilde{\beta} \|_2 
\le \lambda_1 \| \beta^* \|_1 + \lambda_2 \| \beta^* \|_2^2.\label{t2}
\end{equation}
Next, note that 
\begin{equation}
 - [V_n (\tilde{\beta}) - V_n (\beta^*)] = - \sbr[3]{\frac{1}{n} \sum_{i=1}^n (\rho_{\tilde{f}}(Z_i) - E \rho_{\tilde{f}}(Z_i))
-\frac{1}{n} \sum_{i=1}^n (\rho_{f^*}(Z_i) - E \rho_{f^*}(Z_i))}.\label{t2a}
\end{equation}
and recall that $\Xi^* = n^{-1} \sum_{i=1}^n (E \rho_{f^*} - E \rho_{f^0}).\label{t2b}$
Adding $- [V_n (\tilde{\beta}) - V_n (\beta^*)]$ and 
$\Xi^*$ to both sides of (\ref{t2})
\begin{equation}
- [V_n (\tilde{\beta}) - V_n (\beta^*)] + [ P_n \rho_{\tilde{f}} - P_n \rho_{f^*} ] + \Xi^*
+ \lambda_1 \| \tilde{\beta} \|_1 + \lambda_2 \| \tilde{\beta} \|_2^2 \le 
- [V_n (\tilde{\beta}) - V_n (\beta^*)] + \Xi^* + \lambda_1 \| \beta^* \|_1 + \lambda_2 \| \beta^* \|_2^2.\label{t3}
\end{equation}
Now note that by (\ref{t2a}) and the definition of $\Xi^*$ on gets
\[ - [ V_n (\tilde{\beta}) - V_n (\beta^*) ] + [ P_n \rho_{\tilde{f}} - P_n \rho_{f^*}] + \Xi^* = \tilde{\Xi}.\]
Using this in (\ref{t3}) gives
\begin{equation}
\tilde{\Xi} + \lambda_1 \| \tilde{\beta} \|_1 + \lambda_2 \| \tilde{\beta} \|_2^2 \le 
-[V_n (\tilde{\beta}) - V_n (\beta^*)] + \lambda_1 \| \beta^* \|_1 + \lambda_2 \| \beta^* \|_2^2 + \Xi^*.\label{t4}
\end{equation}
Next, we bound $- [ V_n (\tilde{\beta}) - V_n (\beta^*) ]$. To do so we first note that
\begin{align*}
 \| \tilde{\beta} - \beta^* \|_1 = \| t ( \hat{\beta} - \beta^*) \|_1 
 =  \frac{M^*}{M^* + \| \hat{\beta} - \beta^* \|_1} \|  \hat{\beta} - \beta^*\|_1
  \le  M^*.
\end{align*}
Hence,
\begin{align*}
- [ V_n (\tilde{\beta}) - V_n (\beta^*) ] \le \sup_{ \| \beta - \beta^* \|_1 \le M^*} | V_n (\beta ) - V_n (\beta^*)|:=Z_{M^*}.
\end{align*} 
So on $\tau= \{ Z_{M^*} \le \lambda_0 M^* \}$ one may rewrite (\ref{t4}) as
\begin{equation}
\tilde{\Xi} + \lambda_1 \| \tilde{\beta} \|_1 + \lambda_2 \| \tilde{\beta} \|_2^2 \le 
\lambda_0 M^* + \lambda_1 \| \beta^* \|_1 + \lambda_2 \| \beta^* \|_2^2 + \Xi^*.\label{t5}
\end{equation}
Now subtract $\lambda_1\enVert[0]{\tilde{\beta}_{S^*}}_1$ from both sides of (\ref{t5}). Using that $\enVert[0]{\tilde{\beta}}_1-\enVert[0]{\tilde{\beta}_{S^*}}_1=\enVert[0]{\tilde{\beta}_{S^*_c}}_1$ and the continuity of the norm to conclude that $\enVert[0]{\beta^*}_1-\enVert[0]{\tilde{\beta}_{S^*}}_1\leq \enVert[0]{\beta^*-\tilde{\beta}_{S^*}}_1=\enVert[0]{\beta^*_{S^*}-\tilde{\beta}_{S^*}}_1$ one gets
%
\begin{equation}
\tilde{\Xi} + \lambda_1 \| \tilde{\beta}_{S_c^*} \|_1+\lambda_2 \| \tilde{\beta} \|_2^2 
\le 
\lambda_0 M^* + \lambda_1 \|\tilde{\beta}_{S^*} - \beta_{S^*} \|_1 + \lambda_2 \| \beta^* \|_2^2 + \Xi^*.\label{t10} 
\end{equation}
Furthermore, by continuity of the norm,
\begin{equation}
0\leq\envert[1]{\| \beta^*_{S^*}\|_2  - \| \tilde{\beta}_{S^*} - \beta^*_{S^*} \|_2\|}\leq \enVert[0]{\tilde{\beta}_{S^*}}_2\leq \enVert[0]{\tilde{\beta}}_2\label{t11}
\end{equation}
Squaring (\ref{t11}) yields,
\begin{equation}
\| \tilde{\beta}\|_2^2\geq \| \tilde{\beta}_{S^*} \|_2^2 \ge \|\beta^*_{S^*}\|_2^2 + \| \tilde{\beta}_{S^*} - \beta^*_{S^*} \|_2^2 - 2 \|\beta^*_{S^*}\|_2 \|\tilde{\beta}_{S^*}
- \beta^*_{S^*}\|_2.\label{t11a}
\end{equation}
Using (\ref{t11a}) on the left hand side of (\ref{t10}) and rearranging yields 
\[
\tilde{\Xi} + \lambda_1 \| \tilde{\beta}_{S_c^*} \|_1 + \lambda_2 \| \beta^*_{S^*} \|_2^2 + \lambda_2  \| \tilde{\beta}_{S^*} - \beta^*_{S^*} \|_2^2  \le 
\lambda_0 M^* + \lambda_1 \| \tilde{\beta}_{S^*} - \beta^*_{S^*} \|_1 + \lambda_2 \| \beta^* \|_2^2 + 2 \lambda_2 \|\beta^*_{S^*}\|_2 
\| \tilde{\beta}_{S^*} - \beta^*_{S^*} \|_2+  \Xi^*.\]
Noting that $\|\beta^*_{S^*} \|_2 = \| \beta^* \|_2$, the inequality above simplifies to
\begin{equation}
\tilde{\Xi} + \lambda_1 \| \tilde{\beta}_{S_c^*} \|_1 + \lambda_2  \| \tilde{\beta}_{S^*} - \beta^*_{S^*} \|_2^2  \le 
\lambda_0 M^* + \lambda_1 \| \tilde{\beta}_{S^*} - \beta^*_{S^*} \|_1 + 2 \lambda_2 \| \beta^*\|_2 \| \tilde{\beta}_{S^*} -  \beta^*_{S^*} \|_2 + \Xi^*.\label{t12}
\end{equation}
We shall refer to (\ref{t12}) as the Basic Inequality. In the next steps we use the adaptive restricted eigenvalue condition and the margin condition to rewrite the Basic Inequality. We consider 2 cases, which will constitute steps 2 and 3. They depend on the behaviour of $\lambda_1 \| \tilde{\beta}_{S^*} - \beta^*_{S^*} \|_1 + 2 \lambda_2
\| \beta^* \|_2  \| \tilde{\beta}_{S^*} - \beta^*_{S^*} \|_2$.\\

{\bf Step 2}. We consider the case of 
\begin{equation}
\lambda_1 \| \tilde{\beta}_{S^*} - \beta^*_{S^*} \|_1 + 2 \lambda_2  \|\beta^* \|_2 \| \tilde{\beta}_{S^*} - \beta^*_{S^*} \|_2 \ge \Delta^* = \lambda_0 M^*.\label{step2} 
\end{equation}
while Step 3 considers the reverse inequality. First, it will be  shown that $
 \| \tilde{\beta}_{S_c^*} \|_1   \le L_n
 \| \tilde{\beta}_{S^*} - \beta_{S^*} \|_1$
which allows us to use the adaptive restricted eigenvalue condition. To this end note that from the definition of the oracle rate $\lambda_0M^*=\Delta^* \ge \Xi^*$ since $H(\cdot)\geq 0$. Using this and $\tilde{\Xi} \ge 0$ in (\ref{t12}) yields
\begin{equation}
\lambda_1 \| \tilde{\beta}_{S_c^*} \|_1   \le 
2 \Delta^*  + \lambda_1 \| \tilde{\beta}_{S^*} - \beta^*_{S^*} \|_1 + 2 \lambda_2 \| \beta^*\|_2 \| \tilde{\beta}_{S^*} -  \beta^*_{S^*} \|_2.\label{t13}
\end{equation}
Use (\ref{step2}) to rewrite the right hand side of (\ref{t13}) as 
\begin{equation}
\lambda_1 \| \tilde{\beta}_{S_c^*} \|_1   \le 
3(\lambda_1 \| \tilde{\beta}_{S^*} - \beta^*_{S^*} \|_1 + 2 \lambda_2 \| \beta^*\|_2 \| \tilde{\beta}_{S^*} -  \beta^*_{S^*} \|_2).\label{t16}
\end{equation}
Using $ \lambda_1 \| \tilde{\beta}_{S^*} - \beta^*_{S^*} \|_1 \le \lambda_1 \sqrt{s^*} \| \tilde{\beta}_{S^*} - \beta^*_{S^*} \|_2$ in (\ref{t16}) one gets 
\[ \lambda_1 \| \tilde{\beta}_{S_c^*} \|_1   \le 
3(\lambda_1 \sqrt{s^*}  + 2 \lambda_2 \| \beta^*\|_2 \|)\| \tilde{\beta}_{S^*} -  \beta^*_{S^*} \|_2\]
which implies 
\begin{equation}
 \| \tilde{\beta}_{S_c^*} \|_1   \le L_n
 \| \tilde{\beta}_{S^*} - \beta^*_{S^*} \|_2,\label{t17}
\end{equation}
where $L_n  = 3( \sqrt{s^*} + 2 \frac{\lambda_2}{\lambda_1}  \| \beta^*\|_2)$.
To bound the right hand side of (\ref{t12}) from above note that by the adaptive restricted eigenvalue condition
\begin{eqnarray}
\lambda_1 \| \tilde{\beta}_{S^*} - \beta^*_{S^*} \|_1 +  2 \| \beta^*\|_2 \lambda_2 \| \tilde{\beta}_{S^*} -  \beta^*_{S^*} \|_2 
& \le & \lambda_1 \sqrt{s^*} \| \tilde{\beta}_{S^*} - \beta^*_{S^*} \|_2 +  2 \lambda_2 \| \beta^* \|_2 \| \tilde{\beta}_{S^*} -  \beta^*_{S^*} \|_1 \nonumber \\
& \le & (\lambda_1 \sqrt{s^*} + 2 \lambda_2 \| \beta^*\|_2)   \| \tilde{\beta}_{S^*} - \beta^*_{S^*} \|_2\notag \\
&\leq&
(\lambda_1 \sqrt{s^*} + 2 \lambda_2 \| \beta^*\|_2) \frac{  \| f_{\tilde{\beta}} - f_{\beta^*} \|}{\phi_*}
\label{t19}
\end{eqnarray}
Inserting (\ref{t19}) in (\ref{t12}) yields 
\begin{equation}
\tilde{\Xi} + \lambda_1 \| \tilde{\beta}_{S_c^*} \|_1 + \lambda_2 \| \tilde{\beta}_{S^*} - \beta^*_{S^*}\|_2^2  \le 
\lambda_0 M^* + (\lambda_1 \sqrt{s^*}+ 2 \lambda_2 \|\beta^*\|_2) \frac{ \| f_{\tilde{\beta}_{S^*}} - f_{\beta^*} \|}{\phi_*} + \Xi^*.\label{t21}
\end{equation}
Now add  $\lambda_1 \| \tilde{\beta}_{S^*} - \beta^*_{S^*} \|_1$ to both sides of (\ref{t21})
\begin{equation}
\tilde{\Xi} + \lambda_1 \| \tilde{\beta}_{S_c^*} \|_1 + \lambda_1 \| \tilde{\beta}_{S^*} - \beta^*_{S^*} \|_1 +
\lambda_2 \| \tilde{\beta}_{S^*} - \beta^*_{S^*}\|_2^2 \le 
\lambda_0 M^* + \lambda_1 \| \tilde{\beta}_{S^*} - \beta^*_{S^*} \|_1+(\lambda_1 \sqrt{s^*}+ 2 \lambda_2 \|\beta^*\|_2) \frac{ \| f_{\tilde{\beta}} - f_{\beta^*} \|}{\phi_*} + \Xi^*.\label{t22}
\end{equation}
and note that the second term on the right hand side of (\ref{t22}) can be  bounded from above as follows using the adaptive restricted eigenvalue condition
\begin{equation}
\lambda_1 \| \tilde{\beta}_{S^*} - \beta^*_{S^*} \|_1  \le \lambda_1 \sqrt{s}_* \frac{ \| f_{\tilde{\beta}} - f_{\beta^*} \|}{\phi_*}.\label{t23}
\end{equation}
Using (\ref{t23}) in (\ref{t22}) to get 
\begin{equation}
\tilde{\Xi} + \lambda_1 \| \tilde{\beta}_{S_c^*} \|_1 + \lambda_1 \| \tilde{\beta}_{S^*} - \beta^*_{S^*} \|_1 + 
\lambda_2 \| \tilde{\beta}_{S^*} - \beta^*_{S^*}\|_2^2 \le 
\lambda_0 M^* +( 2 \lambda_1 \sqrt{s^*} + 2 \lambda_2 \|\beta^*\|_2)   \frac{ \| f_{\tilde{\beta}} - f_{\beta^*} \|}{\phi_*} + \Xi^*.\label{t24}
\end{equation}
Since we assume  $f^* \in {\bf F}_{\text{local}}$, and since $f_{\tilde{\beta}} \in {\bf F}_{\text{local}}$ (because $\enVert[0]{\tilde{\beta}-\beta^*}\leq M^*$ as shown in Step 1), we get using (\ref{fi}) and the margin condition that
\begin{eqnarray}
(2 \lambda_1 \sqrt{s^*}+ 2 \lambda_2 \| \beta^*\|_2)) \frac{  \| f_{\tilde{\beta}_{S^*}} - f_{\beta^*} \|}{\phi_*} & = &  \frac{ \| f_{\tilde{\beta}_{S^*}} - f_{\beta^*} \|}{2} \left[ \frac{4 \lambda_1 \sqrt{s^*}+ 4 \lambda_2 
\|\beta^*\|_2}{\phi_*} \right] \nonumber \\
& \le & G \left( \frac{ \| f_{\tilde{\beta}} - f_{\beta^*} \|}{2} \right) + H \left( \frac{4 \lambda_1 \sqrt{s^*}+ 4 \lambda_2 \|\beta^*\|_2}{\phi_*} \right) \nonumber \\
& \le & \tilde{\Xi}/2 + \Xi^*/2 + H \left( \frac{4 \lambda_1\sqrt{s^*} + 4 \lambda_2 \|\beta^*\|_2}{\phi_*} \right),\label{t25}
\end{eqnarray}
where we used (\ref{fi}) for the first inequality, with $u =\frac{ \| f_{\tilde{\beta}_{S^*}} - f_{\beta^*} \|}{2}$ and $v =\frac{4 \lambda_1 \sqrt{s^*} + 4 \lambda_2 \|\beta^*\|_2}{\phi_*}$, and the second estimate uses the triangle inequality, convexity of $G(\cdot)$ and the margin condition. Substituting (\ref{t25}) into (\ref{t24}) gives
\begin{equation}
\tilde{\Xi} + \lambda_1 \| \tilde{\beta}_{S_c^*} \|_1 + \lambda_1 \| \tilde{\beta}_{S^*} - \beta^*_{S^*} \|_1 +
 \lambda_2 \| \tilde{\beta}_{S^*} - \beta^*_{S^*} \|_2^2 \le
\lambda_0 M^*+3\Xi^*/2+H \left( \frac{4 \lambda_1\sqrt{s^*} + 4\lambda_2 \|\beta^*\|_2}{\phi_*} \right)+\tilde{\Xi}/2.\label{t26}
\end{equation}
Next, note that  $\| \tilde{\beta}_{S_c^*} \|_1 + \| \tilde{\beta}_{S^*} - \beta^*_{S^*} \|_1 = \| \tilde{\beta} - \beta^* \|_1$ such that (\ref{t26})  can be written as
\[ \tilde{\Xi} + \lambda_1 \| \tilde{\beta} - \beta^* \|_1+ \lambda_2 \| \tilde{\beta}_{S^*} - \beta^*_{S^*} \|_2^2  \le \lambda_0 M^* + 3 \Xi^*/2+H \left( \frac{4 \lambda_1 \sqrt{s^*} + 4\lambda_2 \|\beta^*\|_2}{\phi_*} \right)+\tilde{\Xi}/2.\]
Now we use that $\Delta^*=\lambda_0M^* = 3\Xi^*/2+
H \left( \frac{4 \lambda_1 \sqrt{s^*}+ 4  \lambda_2 \|\beta^*\|_2}{\phi_*} \right)$ to rewrite the above inequality as 
\[ \tilde{\Xi} + \lambda_1 \| \tilde{\beta} - \beta^* \|_1 +\lambda_2 \| \tilde{\beta}_{S^*} - \beta^*_{S^*} \|_2^2 \le 2 \Delta^* + \tilde{\Xi}/2.\]
The above can be rewritten as 
\begin{equation}
\tilde{\Xi}/2 + \lambda_1 \| \tilde{\beta} - \beta^* \|_1 +\lambda_2 \| \tilde{\beta}_{S^*} - \beta^*_{S^*} \|_2^2\le 2 \Delta^*.\label{t27}
\end{equation}
The inequality (\ref{t27}) yields the desired oracle inequality but for $\tilde{\beta}$ instead of $\hat{\beta}$. However, it also follows from (\ref{t27}) (using $\Delta^*=\lambda_0M^*$) that 
\[  \lambda_1 \| \tilde{\beta} - \beta^* \|_1 \le 2 \Delta^* = 2 \lambda_0 M^*\]
which in turn yields (using $\lambda_1 \ge 4 \lambda_0$)
\[ \| \tilde{\beta} - \beta^* \|_1 \le 2 \frac{\lambda_0}{\lambda_1} M^* \le M^*/2.\]
Next, note that by the definitions of $\tilde{\beta}$ and $t$
\begin{eqnarray*}
\tilde{\beta} - \beta^*  =  t \hat{\beta} +(1-t) \beta^* - \beta^* = t (\hat{\beta} - \beta^*) 
 =  \frac{M^*}{M^* + \| \hat{\beta} - \beta^* \|_1 } ( \hat{\beta} - \beta^*).
\end{eqnarray*}
Hence,
\begin{align*}
\frac{M^*}{M^* + \| \hat{\beta} - \beta^* \|_1 } \|\hat{\beta} -\beta^* \|_1 = \| \tilde{\beta} - \beta^* \|_1\leq M^*/2
\end{align*}
which upon rearranging yields  $\| \hat{\beta} - \beta^* \|_1 \le M^*$. But this means that all the above derivations are valid with $\hat{\beta}$ replacing $\tilde{\beta}$ by simply starting from  (\ref{iqaux}) instead of (\ref{conv}). In particular, (\ref{t27}) yields
\begin{equation}
 \Xi (\hat{f}) + 2 \lambda_1 \| \hat{\beta} - \beta^* \|_1 +2 \lambda_2 \| \tilde{\beta}_{S^*} - \beta^*_{S^*} \|_2^2\le 4 \Delta^*= 6 \Xi (f^*) + 4  H \left( \frac{4 \lambda_1 \sqrt{s^*} +  4 \lambda_2 \|\beta^*\|_2}{\phi_*} \right).\label{t2o}
\end{equation}
which implies the bound in Theorem \ref{thm1}.

{\bf Step 3}. Here we consider the case
\begin{equation}
\lambda_1 \| \tilde{\beta}_{S^*} - \beta^*_{S^*} \|_1 + \lambda_2 \|\beta^*\|_2 \| \tilde{\beta}_{S^*} - \beta^*_{S^*} \|_2 < \Delta^*.\label{t31}
\end{equation}
As in Step 2 we note that $\Delta^* \ge \Xi^*$ by the definition of $\Delta^*$. Using also that $\lambda_0 M^* = \Delta^*$ the Basic Inequality in (\ref{t12}) can be written as,
\begin{equation}
\tilde{\Xi} + \lambda_1 \| \tilde{\beta}_{S_c^*} \|_1 +\lambda_2 \| \tilde{\beta}_{S^*} - \beta^*_{S^*} \|_2^2
 \le \lambda_0 M^* + \Delta^* + \Xi^* \le 
\lambda_0 M^* + 2 \Delta^* = 3 \Delta^*.\label{t32}
\end{equation} 
Then adding $\lambda_1 \| \tilde{\beta}_{S^*} - \beta^*_{S^*} \|_1$ to both sides, and noting that 
$\|\tilde{\beta}_{S_c^*}\|_1 + \| \tilde{\beta}_{S^*} - \beta^*_{S^*} \|_1 
= \| \tilde{\beta} - \beta^* \|_1$ one gets
\begin{align}
\tilde{\Xi} + \lambda_1 \| \tilde{\beta} - \beta^* \|_1 + \lambda_2 \| \tilde{\beta}_{S^*} - \beta^*_{S^*} \|_2^2\le 3 \Delta^* + \lambda_1 \| \tilde{\beta}_{S^*} - \beta^*_{S^*} \|_1
\leq 4\Delta^*\label{t32a}.
\end{align}
where the last inequality follows from (\ref{t31}). This is the desired oracle inequality but for $\tilde{\beta}$ instead of $\hat{\beta}$. However, it also follows from (\ref{t32a}) that 
\begin{equation}
\lambda_1 \| \tilde{\beta} - \beta^* \|_1 \le 4 \Delta^*.\label{t33}
\end{equation}
such that
\[ \| \tilde{\beta} - \beta^* \|_1 \le \frac{4 \Delta^*}{\lambda_1} = \frac{ 4 \lambda_0 M^*}{\lambda_1} \le M^*/2,\]
by $\lambda_0 M^* = \Delta^*$ and $\lambda_1 \ge 8 \lambda_0$. By the same arguments as in step 2 it now follows that $\enVert[0]{\hat{\beta}-\beta^*}_1\leq M^*$. This implies that all the above arguments can be repeated with $\hat{\beta}$ replacing $\tilde{\beta}$. In particular, (\ref{t32a}) yields
\begin{equation}
\Xi (\hat{f}) + \lambda_1 \| \hat{\beta} - \beta^* \|_1 +\lambda_2 \| \tilde{\beta}_{S^*} - \beta_{S^*} \|_2^2\le 4 \Delta^*.\label{t34}
\end{equation}
which is the desired inequality. 
\end{proof}

\begin{proof}[{\bf Proof of Theorem \ref{pbound}}]
Set
\begin{align}
\zeta =  D\sbr[4]{ 4 \Lambda \del[2]{\frac{K}{3},n,p} + \frac{tK}{3n} + \sqrt{\frac{2t}{n}} \sqrt{ 1+ 8 \Lambda \del[2]{\frac{K}{3},n,p}}}\label{vdGBbound}
\end{align}
with $\Lambda (\frac{K}{3},n,p) =  \sqrt{\frac{2\log 2p}{n}} + \frac{K\log 2p}{3n}$. Then, for all $t>0$ \cite{vdGB11} show (Theorem 14.5) that
\begin{align*}
P(Z_{M}\leq M\zeta)\geq 1-\exp(-t).
\end{align*}
Next note that there exist a constant $c>0$ (whose value may change throughout the display below) such that
\begin{align*}
\Lambda \del[2]{\frac{K}{3},n,p}
\leq
c\del[3]{\sqrt{\frac{\log(2)}{n}}+\sqrt{\frac{\log(p)}{n}}+\frac{\log(2)}{n}+\frac{\log(p)}{n}}
\leq
c\sqrt{\frac{\log(p)}{n}}
\end{align*}
Hence, choosing $t=\log(p)$ implies that there exists a constant $\tilde{c}>0$ (whose value may change throughout the display below) such that
\begin{align}
\zeta 
&\leq
\tilde{c}D\sbr[4]{\sqrt{\frac{\log(p)}{n}}+\frac{\log(p)}{n}+\sqrt{\frac{\log(p)}{n}}+\del[3]{\frac{\log(p)}{n}}^{3/4}}\\
&\leq
\tilde{c}D\sqrt{\frac{\log(p)}{n}}:=\lambda_0  
\end{align}
This implies
\begin{align*}
P(\tau)
=
P(Z_{M^*}\leq \lambda_0M^*)
\geq
P(Z_{M^*}\leq \zeta M^*)
\geq 
1-\exp\del[1]{-\log(p)}
=
1-\del[2]{\frac{1}{p}}
\end{align*}
\end{proof}

\begin{proof}[Proof of Theorem \ref{thm3}]
The bound in Theorem \ref{thm1} is valid on the set $\tau$. Theorem \ref{pbound} provides the stated lower bound on the probability of $\tau$. Combining these two results gives the theorem.

\end{proof}

\begin{lemma}\label{Hprop}
Let $G(u)$ be a strictly convex function on $[0,\infty)$, with $G(0)=0$. The convex conjugate $H(v)=\sup_{u\geq 0}[uv-G(u)],\ v\geq 0$ satisfies
\begin{enumerate}
\item $H$ is non-negative and non-decreasing.
\item $H$ is convex.
\item $H$ is right-continuous at $0$.
\end{enumerate} 
\end{lemma}

\begin{proof}
The non-negativity of $H$ follows from $H(v)\geq [0\cdot v-G(0)]=0$ for all $v\geq 0$. Let $0\leq v_1\leq v_2$. Then, since $[uv_1-G(u)]\leq [uv_2-G(u)]$ for all $u\geq 0$,  
\begin{align*}
H(v_1)=\sup_{u\geq 0}[uv_1-G(u)]\leq \sup_{u\geq 0}[uv_2-G(u)]=H(v_2)
\end{align*}
and so $H$ is non-deceasing.

The convexity of $H$ may be found in Theorem 12.2 of \cite{rockafellar97}. Here, for the sake of completeness, we give a more direct argument. For any $0<\lambda<1$  and $v_1,v_2\geq 0$
\begin{align*}
H(\lambda v_1+(1-\lambda)v_2)
&=
\sup_{u\geq 0}[u(\lambda v_1+(1-\lambda)v_2)-G(u)]
=
\sup_{u\geq 0}[\lambda(uv_1-G(u))+(1-\lambda)(uv_2-G(u))]\\
&\leq
\lambda\sup_{u\geq 0}(uv_1-G(u))+(1-\lambda)\sup_{u\geq 0}(uv_2-G(u))
=
\lambda H(v_1)+(1-\lambda)H(v_2)
\end{align*}
establishing the convexity of $H$.

To establish that $H$ is right continuous at 0 note first that $H(0)$ is a lower bound for  $\cbr[0]{H(x_n)}_{n=1}^\infty$ for any sequence $x_n\downarrow 0$ since $H$ is non-decreasing. Hence, $\cbr{H(x_n)}_{n=1}^\infty$ is a bounded non-increasing sequence and so it possesses at limit which furthermore satisfies $H(0)\leq \inf_{n}H(x_n)=\lim_n H(x_n)$. It suffices to show that $H(0)\geq \inf_{x>0}H(x)=\inf_nH(x_n)$ to conclude  $H(0)=\inf_{n}H(x_n)=\lim_n H(x_n)$. We assume the converse to reach a contradiction., i.e. assume that $H(0)< \inf_{x>0}H(x)$. In particular, $H(0)<\inf_{0<\lambda<1}H((1-\lambda)x)$ for all $x>0$ such that there exists an $\epsilon>0$ satisfying $\inf_{0<\lambda<1}H((1-\lambda)x)=H(0)+\epsilon$. But by the convexity of $H$ it holds for all $0<\lambda<1$ that
\begin{align*}
\lambda H(0)+(1-\lambda)H(x)
\geq 
H((1-\lambda)x)
\geq
H(0)+\epsilon
\end{align*}  
By continuity of the left hand side in $\lambda$ it follows that
\begin{align*}
H(0)
=
\lim_{\lambda \uparrow 1}\sbr[1]{\lambda H(0)+(1-\lambda)H(x)}
\geq 
H(0)+\epsilon
\end{align*}
which is a contradiction and so we can't have $H(0)<\inf_{x>0}H(x)$ and we conclude that $H$ is right-continuous at 0.
\end{proof}

\begin{proof}[{\bf Proof of Corollary \ref{corloss}}]
First, note that the probability with which the inequality in Theorem \ref{thm3} is valid tends to one. Also, this inequality implies
\begin{align*}
\Xi (\hat{f})
\leq
6 \Xi (f^*) + 4  H \left( \frac{4 \lambda_1 \sqrt{s^*} +4 \lambda_2 \|\beta^*\|_2}{\phi_*} \right)
\end{align*}
Hence, it suffices to show that $\limsup_{n\to \infty}H \left( \frac{4 \lambda_1 \sqrt{s^*} +4 \lambda_2 \|\beta^*\|_2}{\phi_*} \right)=\limsup_{n\to \infty}H \left( \frac{6\lambda_1 \sqrt{s^*}}{\phi_*} \right)=0$. To this end, observe that with $\lambda_1 \in O( \lambda_0)$
\begin{align*}
\frac{6\lambda_1 \sqrt{s^*}}{\phi_*} \in  O\del[2]{\sqrt{\frac{n^{a}}{n}n^b}}=O\del[1]{n^{a/2+b/2-1/2}}
\subseteq
o(1)
\end{align*}
which yields the desired result by the-right continuity of $H$ established in Lemma \ref{Hprop}. 
\end{proof}

\begin{proof}[{\bf Proof of Corollary \ref{corlin}}]
The first two inequalities, (\ref{loss1}) and (\ref{lin1}), follow from Theorem \ref{thm3} upon using the same reasoning as in remark 2 preceding Theorem \ref{thm1}.  In particular, choose $\Gamma=S^0$ in the definition of the oracle. This implies $\Xi(f^*)=\Xi(f^0)=0$ (as seen in remark 2). (\ref{loss2}) and (\ref{lin2}) follows from (\ref{loss1}) and (\ref{lin1}) under the given assumptions by simple insertion.
\end{proof}

\begin{proof}[{\bf Proof of Corollary \ref{corlinasym}}]
First note that the probability with which inequality (\ref{lin2}) is valid tends to one. It remains to be shown that the right hand side of (\ref{lin2}) tends to zero. But under the stated conditions the right hand side is of order
\begin{align*}
O\left( \sqrt{\frac{n^a}{n}} n^b \right)=O\del[2]{n^{a/2+b-1/2}}\subseteq o(1)
\end{align*}
where the last inclusion follows from the assumption $a+2b<1$.
\end{proof}

\begin{proof}[Proof of Lemma \ref{quad}]

a)
{\it The Analysis of Assumption 1}.
First, note that by the second derivative test $\rho(f(x),y)$ is a convex function in $f(x)$. Since, the second derivative is constant, and equal to 2 this also shows that the margin condition is satisfied with a quadratic margin and $2c=2$, i.e. $c=1$. This was of course already clear from (\ref{qm}) prior to Lemma \ref{quad}. The analysis of assumption 2 is slightly more involved:

{\it The Analysis of Assumption 2}.
We show that $f_\beta\in\mathbf{F}_{\text{local}}$ under the stated conditions. More precisely, we must show that $\enVert[0]{f_\beta-f^0}_\infty\leq \eta$. Since
\begin{align*}
\enVert[0]{f_\beta-f^0}_\infty
\leq
\enVert[0]{f_\beta-f^*}_\infty+\enVert[0]{f^*-f^0}_\infty 
\leq
\enVert[0]{f_\beta-f^*}_\infty+\eta/2
\end{align*}
it suffices to show that $\enVert[0]{f_\beta-f^*}_\infty\leq \eta/2$. To this end, note that
\begin{align*}
\envert[0]{f_\beta(x)-f^*(x)}
=
\envert[2]{\sum_{j=1}^p(\beta_j-\beta_j^*)\psi_j(x)}
\leq
\enVert{\beta-\beta^*}_{1}\max_{1\leq j\leq p}\envert[1]{\psi_j(x)}
\end{align*}
which implies $\enVert[0]{f_\beta-f^*}_\infty\leq M^*K$. Hence, it suffices to show that $M^*K\leq \eta/2$. To do so, recall that by using $H(v)=v^2/4c$ (with $c=1$), $\Xi(f^*)=E(f^*-f^0)^2$ and $\lambda_2 = \frac{\lambda_1 \sqrt{s^*}}{2 \| \beta^*\|_2}$ 
\begin{align*}
M^*
&=
\frac{\Delta^*}{\lambda_0}
= 
\frac{1}{\lambda_0}\del[4]{(3/2) E(f^*-f^0)^2+H \left(\frac{4 \lambda_1 \sqrt{s^*} + 4 \lambda_2 \|\beta^*\|_2}{\phi (S^*)}\right)}\\
&=
\frac{1}{\lambda_0}\del[4]{(3/2)E(f^*-f^0)^2+\left(\frac{4 \lambda_1 \sqrt{s^*} + 4 \lambda_2 \|\beta^*\|_2}{\phi (S^*)}\right)^2/4}\\
&=
\frac{1}{\lambda_0}\del[4]{(3/2) E(f^*-f^0)^2+\frac{9 \lambda_1^2 s^*}{\phi^2(S^*)}}\\
&\leq
\frac{1}{\lambda_0}\del[4]{\frac{18 \lambda_1^2 s^*}{\phi^2(S^*)}}\\
&\leq
\frac{18 L^2\lambda_0 s^*}{\phi^2(S^*)}
\end{align*}
such that $M^*K\leq\eta/2$ under the stated assumptions.  

{\it The Analysis of Assumption 3}. The validity of Assumption 3 follows from the fact that $\Sigma$ is assumed to have full rank. This is sufficient for Assumption 3 to be valid as argued just after (\ref{reseig}) in Section \ref{setup}.

Inequality (\ref{quada}) follows upon using $H(v)=v^2/4c$ with $c=1$ as well as $\Xi(f)=E(f-f^0)^2$ for all $f\in \mathbf{F}$ in Theorem \ref{thm1}.

b)
Next, we turn to part b) of the lemma. This result will be derived based on Theorem \ref{pbound}. Hence, we verify the assumptions of that theorem. The two boundedness conditions are valid by assumption. Next, we establish the local Lipschitz continuity. Note that on $\mathcal{A}=\cbr[0]{\max_{1\leq i\leq n}|X_i| \vee |\epsilon_i|\leq C_n}$
\begin{align*} 
\envert[3]{\frac{\partial \rho(f_\beta(X_i),Y_i)}{\partial f_\beta(X_i)}}
&=
2\envert[0]{Y_i-f_\beta(X_i)}
=
2(\envert[0]{\epsilon_i+f^0(X_i)-f_\beta(X_i)})
\leq 
2\del[3]{\envert[0]{\epsilon_i}+F_{C_n}+\envert[2]{\sum_{j=1}^p\beta_j\psi_j(x)}}\\
&\leq
2\del[1]{\envert[0]{\epsilon_i}+F_{C_n}+\enVert[0]{\beta}_1\max_{1\leq j\leq p}\enVert{\psi_j}_\infty}
\leq
2(C_n+F_{C_n}+GK)
\end{align*}
for all $i=1,...,n$. So, on the set $\mathcal{A}$, the first derivative of the loss function is bounded and hence the loss function is Lipschitz continuous on this set. This implies that
\begin{align*}
P(\tau^c)\leq P(\tau^c\cap \mathcal{A})+P(\mathcal{A}^c)
\end{align*}
By the above arguments $\rho(f(x),y)$ is Lipschitz continuous on $\mathcal{A}$ with Lipschitz constant $2(C_n+F_{C_n}+GK)$. Hence, by Theorem \ref{pbound} $P(\tau^c\cap \mathcal{A})\leq \del[2]{\frac{1}{p}}$ and the Lipschitz constant $D_n$ in the definition of $\lambda_0=dD_n\sqrt{\frac{\log(p)}{n}}$ in Theorem \ref{pbound} may be taken to be $2(C_n+F_{C_n}+GK)$. Next, through subgaussianity of $X_1, \epsilon_1$, by a union bound it follows that $P(\mathcal{A}^c)\leq 2\alpha n\exp(-\delta C_n^2)$ for positive constants $\alpha$ and $\delta$. This yields the stated lower bound on the probability of $\tau$ (on which inequality (\ref{quada}) is valid).
\end{proof}

\begin{proof}[Proof of Corollary \ref{nonpar}]
It follows directly from Lemma \ref{quad} that
\begin{align}
E(\hat{f}-f^0)^2
\leq
6E(f^*-f^0)^2 +  \frac{36 L^2\lambda_0^2 s^*}{\phi_*^2}\label{s1}.
\end{align}
with probability at least $1-\del[2]{\frac{1}{p}}-2\alpha n\exp(-\delta C_n^2)$.

The assumptions of part b) are sufficient for Lemma \ref{quad} to be valid as seen in the discussion succeeding Lemma \ref{quad}. (\ref{seriesest2}) follows from (\ref{seriesest}) using (\ref{scond}). The estimate (\ref{seriesest3}) follows by minimizing the order of the upper bound in (\ref{seriesest2}) with respect to $s^*$. The equality in (\ref{seriesest3}) uses that  $\lambda_0=dD_n\sqrt{\frac{\log(p)}{n}}$ with $D_n=2(C_n+F_{C_n}+GK)$ and $d>0$ and that $F_{C_n}$ is bounded since $f^0$ is continuous and the covariates have compact support. (\ref{seriesest4}) follows by insertion of $C_n$ into \ref{seriesest3}.
\end{proof}

\begin{proof}[Proof of Corollary \ref{quadcor}]
First note that the choice of $C_n$ and $p\to\infty$ ensure that the probability with which inequality (\ref{quada}) is valid tends to one. 
%
To proof part a) it suffices to show that $\lambda_0^2s^*\to 0$ (this follows from (\ref{quada})) which is in turns implied by
\begin{align*}
(C_n^2+F_{C_n}^2)\frac{\log(p)}{n}s^*
\in
O\del[2]{(\log(n)+n^{2\tilde{d}})\frac{n^a}{n}n^b}
\subseteq
o(1)
\end{align*}
where the first inclusion is by assumption and the second follows from $a+b+2\tilde{d}<1$.

Regarding part b), choosing $\Gamma=S^0$ in the definition of the oracle implies $\Xi(f^*)=\Xi(f^0)=0$ (as seen in remark 2 after Theorem \ref{thm1}) since the best linear predictor of a linear target is just the target itself (in this case we of course also have $s^*=s^0=1$ and $\beta^*=\beta^0$). Hence, we deduce from (\ref{quada}) that
\begin{align*}
|\hat{\beta}-\beta^0|
=
\| \hat{\beta} - \beta^0 \|_1 
\leq
\frac{9 L^2\lambda_0}{2\phi_*^2}.
\end{align*}

with probability tending to one for $\lambda_0=dD_n\sqrt{\frac{\log(p)}{n}}$ with $D_n=2(C_n+F_{C_n}+GK)$ and $d>0$. So, it suffices to see that $\lambda_0\to 0$ which is implied by
\begin{align*}
(C_n^2+F_{C_n}^2)\frac{\log(p)}{n}
\in
O\del[2]{(\log(n)+n^{2\tilde{d}})\frac{n^a}{n}}
\subseteq
o(1)
\end{align*}
if $a+2\tilde{d}<1$.
\end{proof}

\begin{proof}[Proof of Lemma \ref{logit}] 
The proof is similar to the one of Lemma 6.8 in \cite{vdGB11}. 
First, note that $\rho(f(x),y)$ is a convex function in $f(x)$ since it can be written as the sum of convex functions: the first right hand side term in (\ref{5.1}) of $\rho(f(x),y)$ is linear in $f(x)$, and the second term has positive second derivative. We start by showing that Assumptions 1-3 are satisfied.

{\bf Step 1}. {\it The Analysis of Assumption 1}. We shall show that one may choose $G(x)=cx^2$ for some positive constant $c$ to be defined precisely below. To do so we follow the general route laid out in the beginning of Section \ref{Examples}.
Define
\begin{align}
l(f(x),x) = E [ \rho(f(X),Y)|(X,f(X))=(x,f(x))] = - \pi(x)f(x) + \log (1 + \exp (f(x)))\label{log}
\end{align}

where $\pi(x)=E(Y|(X,f(X))=(x,f(x))$. (\ref{log}) is minimized with respect to $f\in\mathbf{F}$ at $f^0(x)=\log \left( \frac{\pi(x)}{1 - \pi(x) }\right)$. Hence, $f^0 (x) = \log \left( \frac{\pi(x)}{1 - \pi(x) }\right)$. Note that the second order partial derivative of $l(f(x),x)$ with respect to $f(x)$ is
\begin{equation}
\eval[2]{\frac{\partial^2 l (a,x)}{\partial a^2}}_{a=f(x)} 
= 
\frac{\exp(f(x))}{1+\exp(f(x))} \left( 1 - \frac{\exp(f(x))}{1+\exp(f(x))} \right)
=
\frac{\exp(f(x))}{(1+\exp(f(x)))^2}.\label{5.2} 
\end{equation}
So we must show that $\frac{\exp(f(x))}{(1+\exp(f(x)))}$ is bounded from below by a constant for $f\in \mathbf{F}_{\text{local}}$. To do so it suffices to bound $f(x)$ from above and below. To this end, note that for all $x\in\mathcal{X}$
\begin{equation*}
f^0(x)-\enVert[0]{f-f^0}_{\infty}\leq f(x)\leq f^0(x)+\enVert[0]{f-f^0}_{\infty}
\end{equation*}
which implies that for all $f\in\mathbf{F}_{\text{local}}$
\begin{align}
f^0(x)-\eta \leq f(x)\leq f^0(x)+\eta \label{pl2.2}
\end{align}
Furthermore, since $f^0(x) = \log \left( \frac{\pi(x)}{1 - \pi (x)} \right)$ and $\epsilon_0\leq \pi(x)\leq 1-\epsilon_0$, we get 
\begin{align*}
\log\del[2]{\frac{\epsilon_0}{1-\epsilon_0}}\leq f^0(x)\leq \log\del[2]{\frac{1-\epsilon_0}{\epsilon_0}}.
\end{align*} 
Together with (\ref{5.2}) and (\ref{pl2.2}) this implies that
\begin{align*}
\eval[2]{\frac{\partial^2 l (a,x)}{\partial a^2}}_{a=f(x)} 
\geq
\frac{\frac{\epsilon_0}{1-\epsilon_0}e^{-\eta}}{\del[1]{1+\frac{1-\epsilon_0}{\epsilon_0}e^\eta}^2}
\geq 
\frac{\epsilon_0e^{-\eta}}{\del[1]{1+e^\eta/\epsilon_0}^2}>0 
\end{align*}
for all $x\in\mathcal{X}$. Hence, one may use $2c=\frac{\epsilon_0e^{-\eta}}{\del[1]{1+e^\eta/\epsilon_0}^2}$. 

{\bf Step 2}.{\it The Analysis of Assumption 2}. We show that $f_\beta\in\mathbf{F}_{\text{local}}$ under the stated conditions. More precisely, we must show that $\enVert[0]{f_\beta-f^0}_\infty\leq \eta$. Since
\begin{align*}
\enVert[0]{f_\beta-f^0}_\infty
\leq
\enVert[0]{f_\beta-f^*}_\infty+\enVert[0]{f^*-f^0}_\infty 
\leq
\enVert[0]{f_\beta-f^*}_\infty+\eta/2
\end{align*}
it suffices to show that $\enVert[0]{f_\beta-f^*}_\infty\leq \eta/2$. To this end, note that
\begin{align*}
\envert[0]{f_\beta(x)-f^*(x)}
=
\envert[2]{\sum_{j=1}^p(\beta_j-\beta_j^*)\psi_j(x)}
\leq
\enVert{\beta-\beta^*}_{1}\max_{1\leq j\leq p}\envert[1]{\psi_j(x)}
\end{align*}
which implies $\enVert[0]{f_\beta-f^*}_\infty\leq M^*K$. Hence, it suffices to show that $M^*K\leq \eta/2$. To do so, recall that, by using $H(v)=v^2/4c$, and $\lambda_2 = \frac{\lambda_1 \sqrt{s^*}}{2 \| \beta^*\|_2}$ 
\begin{align*}
M^*
&=
\frac{\Delta^*}{\lambda_0}
= 
\frac{1}{\lambda_0}\del[4]{(3/2) \Xi (f^*)+H \left(\frac{4 \lambda_1 \sqrt{s^*} + 4 \lambda_2 \|\beta^*\|_2}{\phi (S^*)}\right)}\\
&=
\frac{1}{\lambda_0}\del[4]{(3/2) \Xi (f^*)+\left(\frac{4 \lambda_1 \sqrt{s^*} + 4 \lambda_2 \|\beta^*\|_2}{\phi (S^*)}\right)^2/4c}\\
&=
\frac{1}{\lambda_0}\del[4]{(3/2) \Xi (f^*)+\frac{9 \lambda_1^2 s^*}{\phi^2(S^*)c}}\\
&\leq
\frac{1}{\lambda_0}\del[4]{\frac{18 \lambda_1^2 s^*}{\phi^2(S^*)c}}\\
&\leq
\frac{18 L^2\lambda_0 s^*}{\phi^2(S^*)c}
\end{align*}
such that $M^*K\leq\eta/2$ under the stated assumptions. 

{\bf Step 3}.{\it The Analysis of Assumption 3}. The validity of Assumption 3 follows from the fact that $\Sigma$ is assumed to have full rank. This is sufficient for Assumption 3 to be valid as argued just after (\ref{reseig}) in Section \ref{setup}.

Inequality (\ref{logitbound1}) follows from Theorem 1 upon using $H(v)=v^2/4c$ with $c=\frac{\epsilon_0e^{-\eta}}{2\del[1]{1+e^\eta/\epsilon_0}^2}$. 
 
Next, we turn to part b) of Lemma \ref{logit}. It suffices to verify the assumptions of Theorem \ref{pbound} since (\ref{logitbound1}) is valid on the set $\tau$. First, note that $\rho(f(x),y)$ is Lipschitz continuous in $f(x)$ for all $y\in\mathcal{Y}$ since
\begin{align*}
\envert[3]{\eval[2]{\frac{\partial \rho(a,y)}{\partial a}}_{a=f(x)}}
=
\envert[1]{-y+\frac{e^{f(x)}}{1+f(x)}}
\leq
2
\end{align*}
and so $D$ in \ref{lipschitz} may be chosen to be 2. The two boundedness assumptions on the basis functions are valid by assumption. So, using $H(v)=v^2/4c$ in Theorem \ref{thm3} yields
\begin{align*}
\Xi (\hat{f}) + \lambda_1 \| \hat{\beta} - \beta^* \|_1 + \lambda_2 \|\hat{\beta}_{S^*} - \beta_{s^*}\|_2^2
&\leq 
6 \Xi (f^*) + 4  H \left( \frac{4 \lambda_1 \sqrt{s^*} +4 \lambda_2 \|\beta^*\|_2}{\phi_*} \right)
&\leq
6 \Xi (f^*) + 36 \frac{\lambda_1^2 s^*}{\phi_*^2c}
\end{align*}
\end{proof}

\bibliographystyle{chicagoa}	
\bibliography{references}		
\end{document}